\documentclass[11pt]{amsart} 	
\usepackage[margin=1in]{geometry}   
\usepackage{amsmath}
\usepackage{flexisym}
\usepackage{breqn}             		
\usepackage{yfonts}

\usepackage{thmtools, thm-restate}
\declaretheorem[numberwithin=section]{theorem}

\usepackage{graphicx}				
\usepackage{amssymb}

\usepackage{color}
\usepackage{graphicx}
\usepackage{mathtools}
\usepackage{amsthm}
\usepackage{tikz-cd}
\usepackage{comment}
\usepackage{xcolor}

\newcommand{\ZZ}{\mathbb{Z}}

\newcommand{\FF}{\mathbb{F}}
\newcommand{\SSS}{\mathbb{S}}

\newcommand{\HH}{\mathbb{H}}
\newcommand{\WW}{\mathbb{W}}
\newcommand{\uu}{u_{2}}

\renewcommand{\u}{u_{h-1}}
\newcommand{\F}{\underset{F}+}
\newcommand{\gF}{\underset{g_*F}+}

\usepackage{hyperref}
\definecolor{dark-red}{rgb}{0.6,0.15,0.15}
\definecolor{dark-blue}{rgb}{0.15,0.15,0.6}
\definecolor{medium-blue}{rgb}{0,0,0.5}

\hypersetup{
	colorlinks, 
	linkcolor=dark-red,
	citecolor=dark-blue, urlcolor=medium-blue
}

\usepackage[nameinlink,capitalise,noabbrev]{cleveref}

\numberwithin{equation}{section}

\newtheorem*{thm*}{Theorem}

\newtheorem{lemma}{Lemma}[section]

\theoremstyle{definition}

\newtheorem{note}{Note}[section]

\newtheorem{rem}{Remark}[section]

\makeatletter
\let\c@equation=\c@thm
\let\c@lem=\c@thm
\let\c@theorem=\c@thm
\let\c@lemma=\c@thm
\let\c@Theorem=\c@thm
\let\c@Lemma=\c@thm
\let\c@cor=\c@thm
\let\c@corollary=\c@thm
\let\c@Corollary=\c@thm
\let\c@conj=\c@thm
\let\c@conjecture=\c@thm
\let\c@prop=\c@thm
\let\c@proposition=\c@thm
\let\c@defn=\c@thm
\let\c@definition=\c@thm
\let\c@Definition=\c@thm
\let\c@notation=\c@thm
\let\c@note=\c@thm
\let\c@exmp=\c@thm
\let\c@ex=\c@thm
\let\c@exmps=\c@thm
\let\c@rem=\c@thm
\let\c@warn=\c@thm
\let\c@claim=\c@thm
\let\c@convention=\c@thm
\let\c@conventions=\c@thm
\let\c@quest=\c@thm
\let\c@facts=\c@thm
\makeatother

\usepackage[utf8]{inputenc}
\usepackage[english]{babel}

\title[Morava Stabilizer Group Action]{A Computation of the Action of the Morava Stabilizer Group on the Lubin-Tate Deformation Ring}
\author{Andr\'e Davis}
\thanks{This material is based upon work supported by the National Science Foundation under Grant No. DMS-1906227.}

\begin{document}
	\maketitle
	\tableofcontents
	\begin{abstract}
		We compute  recursive approximations of the action of the height \(h \geq 2\) Morava stabilizer group on the associated Lubin-Tate deformation ring. We then specialize to the case \(h=3\) and \(p>2\) to calculate the action explicitly. These results are new for \(h>2\) and agree with computations  by Lader at height \(h=2\).
	\end{abstract}

	\section{Introduction}
		For a fixed prime \(p\) and an integer \(h \geq 2\), we let \(\HH_h\) denote the Honda formal group law with coefficients in the finite field \(\FF_{p^h}\). We are interested in its group of automorphisms, and in particular how this group acts on the Lubin-Tate deformation ring. This action appears frequently in the literature and has been described by Devinatz and Hopkins \cite{devinatz2004homotopy} as ``the most important group action in the whole chromatic approach to stable homotopy theory".
	
	It is well known that the Lubin-Tate deformation ring \(R_h\) is isomorphic to 
	\[R_h \cong \WW[[u_1,...,\u]]\]
	where \(\WW = W(\FF_{p^h})\) is the Witt vectors over \(\FF_{p^h}\). We denote by \(\SSS_h\) the automorphism group of \(\HH_h\). 
	
	The action of \(\SSS_h\) on \(R_h\) is extremely complicated. For instance, the associated cohomology groups 
	encode deep information about the stable homotopy category (see \cite{devinatz2004homotopy}, for example). In fact, it is difficult even to compute the leading term of the action. The bulk of
	Chai's paper \cite{chai1996group} is devoted to this computation. In this paper, we consider the action of \(\SSS_h\) only on certain elements of \(R_h\).
	More specifically, we look at the induced action on the quotient of \(R_h\) by the invariant ideal \(I_{h-1} = (p,u_1,...,u_{h-2})\). Several efforts in this direction have been made at height \(h=2\). In this case, of course, \(I_1 = (p)\).
	
	In the special case \(h=2\) and \(p=3\), an approximation of the action appears in
	the 2008 paper \cite{henn2013homotopy}. The authors use the Lubin-Tate theory of deformations to encode the action in the coefficients of a certain power series. This approach results in a system of recursive equations that are then solved explicitly.	
	In 2013, two independently-written papers extended the computations at height \(h=2\) to arbitrary primes \(p\).  One was the PhD thesis of Lader \cite{lader2013resolution}, who also used Lubin-Tate theory. See also \cite{beaudry2019computations} for a review of Lader's results in English, among other things.
	The other
	is a paper by Kohlhaase \cite{kohlhaase2013iwasawa}, who made use of the rigid analytic period morphism of Gross and Hopkins \cite{hopkins1994equivariant} as applied to computations by Devinatz-Hopkins in \cite{devinatz1995action}.

	We follow the program outlined by Lader to extend the computations to all heights \(h\) and all primes \(p\). As with the previous efforts, we obtain recursive formulas. We note that when evaluated at height \(h=2\), our formulas agree with previous computations, and provide the same level of accuracy. Finally, we specialize to \(h=3\) and \(p>2\) and solve our recursive equations explicitly. We note that while this final computation is long and tedious, it would not be difficult to generalize  to other small heights \(h\), particularly with the help of computers.

	\section{Organization and Main Results}	
	Our first task is to approximate the universal deformation \(F\) of the Honda formal group law, working at height \(h>2\). This is the main result of Section \ref{computation of F}, where we break the computation  into two parts. We first compute the exponential \(\exp(x)\) of \(F\) as the inverse of \(\log(x)\), the logarithm of \(F\). 	\begin{lemma}[Cf. \ref{exp}]
		Assume \(h >2\). Given 
		\[\log(x) = x + L_{h-1}x^{p^{h-1}} + L_hx^{p^h} \mod( x^{p^{h+1}}),  \]
		the inverse series is given by
		\[\exp(x) = x +\sum_{j=0}^{p-1} \frac{(-1)^{j+1}}{j+1} {p^{h-1}(j+1) \choose j }L_{{h-1}}^{j+1}x^{j(p^{h-1}-1)+p^{h-1}} -L_{h}x^{p^h} \ \mod (x^{p^h+1}).\]
	\end{lemma}
	
	This allows us to compute \(F\) as
	\[F(x,y) = \exp(\log(x) + \log(y)).\]	
	\begin{thm*}[Cf. \ref{F(x,y) computation}.]
		For $h>2$ and $F(x,y)$ as above, modulo $(u_1,...,u_{h-2},(x,y)^{p^{h}+1})$
		\begin{align*}
			F(x,y) &= x +y -\frac{u_{h-1}}{1-p^{p^{h-1}-1}}C_{p^{h-1}}(x,y) - \frac{1}{1-p^{p^h-1}}C_{p^h}(x,y) +\sum_{m=1}^{p-1}\u^{m+1}P_{m}(x,y)
		\end{align*}
		where 
		\[C_{p^n}(x,y) = \frac{1}{p}((x+y)^{p^n} - x^{p^n} -y^{p^n})\]
		and
		\begin{align*}P_{m}(x,y) =  \frac{1}{(p-p^{p^{h-1}})^{m+1}} \sum_{j=0}^{m} \frac{(-1)^{j+1}}{j+1}&{{p^{h-1}(j+1)} \choose j} {{j(p^{h-1}-1)+p^{h-1}} \choose {m-j}}\\
			&(x+y)^{p^{h-1}(j+1)-m}(x^{p^{h-1}} + y^{p^{h-1}})^{m-j}.
		\end{align*}
	\end{thm*}
	An arithmetic coincidence (cf. Remark \ref{h>2 is necessary}) splits the computation into two cases: \(h=2\) and \(h>2\). In particular, \(F\) takes a slightly different form in the case \(h=2\). See \cite{lader2013resolution} Lemma 3.1, or \cite{beaudry2019computations} Theorem 3.1.

	The universal deformation \(F\) is, in particular, a formal group law with coefficients in the Lubin-Tate deformation ring \(R_h\). It is universal in the sense that it induces a bijective correspondence between \(\star\)-isomorphism classes of deformations of \(\HH_h\) to complete local rings and continuous ring maps out of \(R_h\). 
	Roughly speaking, each automorphism of \(\HH_h\) determines a ``change of coordinates" that we may apply to \(F\) to create a new deformation with coefficients in \(R_h\). Each deformation is classified by an automorphism of \(R_h\), and these automorphisms together determine the action of \(\SSS_h\). See \cite{rezk1998notes} for an overview.

	Having an explicit formula for \(F\) allows us to access the action directly. 
	In Section \ref{The Action} we compute the action of an element \(g\in \SSS_h\) in terms of a unit \(t_0\) in \(R_h\). Although we must assume \(h>2\) to use Theorem \ref{F(x,y) computation}, our results agree with previous computations at \(h=2\).
	
	Before we  state our results, we recall that \(\SSS_h\) is isomorphic to the subgroup of units in
	\[ \WW\langle S \rangle / (S^h=p, Sa=a^{\sigma}S )\]
	where \(\WW = W(\FF_{p^h})\) are the Witt vectors, \(a\in \WW\) and \(\sigma\) is a lift of the Frobenius. Therefore 
	any \(g\in \SSS\) can be expressed as
	\[g = \sum_{i\geq 0}g_i S^i.\]

	We prove in Theorem \ref{action} that for all \(h\geq 2\)
	\[g_*(\u) = t_0^{p^{h-1}-1}\u \mod (p,u_1,...,u_{h-2})\]
	for a unit \(t_0\) such that \(t_0 \equiv g_0 \mod (p,u_1,...,u_{h-2},\u)\). We remark that this computation was previously known to the stable homotopy theory community. 
	
	The rest of Section \ref{The Action} is devoted to determining \(t_0\) in terms of the generators of \(R_h\). Several long computations produce the following recursive expressions. 
	\begin{thm*}[Cf. Theorem \ref{tk}, Theorem \ref{th}]
		Let \(g\in \SSS_h\) as above, and let \(h\geq 2\). For \(0\leq k\leq h-2\),
		\[t_k = t_k^{p^h}+\u t_{k+1}^{p^{h-1}}-\u^{p^{k+1}}t_{k+1}t_0^{p^{k+1}(p^{h-1}-1)} \mod I_{h-1}\]
		and
		\begin{dmath*}
			t_{h-1} 	=
			t_{h-1}^{p^{h}} + \u t_h^{p^{h-1}} 
			- \sum_{j=1}^{p-1}\frac{1}{p}{p \choose j}\u^{jp^{h-2}+1}t_1^{jp^{2h-3}}t_0^{p^{2h-2}(p-j)} \mod (I_{h-1},\u^{p^{h-1}+1}).
		\end{dmath*}
	\end{thm*}
	We remark that not only do these expressions agree with the previous computations at height \(h=2\) mentioned above, the accuracy of \(t_{h-1}\) in particular also matches.

	Finally, in Section \ref{h=3} we specialize to \(h=3\). Assuming \(p>2\), we are able to compute the following explicit formula for \(t_0\) as a polynomial in \(u_2\). 
	\begin{thm*}[Cf. Theorem \ref{th=3}]
		Assume \(h=3\) and \(p>2\). Let \(g = 1+g_1S+g_2S^2+g_3S^3 \mod (S^4)\). Then modulo \((p,u_1,\uu^{2p^2+p+1})\),
		\begin{dmath*}
			t_0
			=1+ \uu g_1^{p^2}
			- \sum_{i=0}^{p-1}(-1)^i \uu^{(i+1)p}g_1^{i+1} 
			-\sum_{i=0}^{p-2} (-1)^i\uu^{(i+1)p +1}g_1^ig_2^{p^2}
			+\uu^{p^2+1}\left(g_2^p - g_1^{p-1}g_2^{p^2}\right)
			- \sum_{i=0}^{p}\uu^{p^2+(i+1)p}g_1^{i}\left( (-1)^{p+i}g_1^{p+1} +  {{p^2-2}\choose{i}}g_1^{p+1} + (-1)^ig_2   +{{p^2-1}\choose {i-1}}g_2   \right) 
			- \uu^{p^2+p+1}\left(  g_3^p - g_3^{p^2} 
			\right)	
			-\sum_{j=1}^{p-1}\uu^{p^2+(j+1)p+1}g_1^{j-1} \left[C_j +\sum_{i=1}^j \frac{1}{p}{{p}\choose{i}}(-1)^{j-i}g_1   \right],
		\end{dmath*}	
		where
		\begin{dmath*}
			C_j =  (-1)^{p+j}g_1^{p+1}g_2^{p^2} + \left( (-1)^j(g_3^p-g_3^{p^2}) + {{p^2-2}\choose{j}}g_1^pg_2^{p^2}  \right)g_1 + {{p^2-2} \choose{j-1}}g_2^{p^2+1}.
		\end{dmath*}
	\end{thm*}
	
	The accuracy \(u_2^{2p^2+p+1}\)  is the highest achievable with our formulas. At arbitrary heights, our formulas theoretically may achieve an accuracy of \(\u^{p^{h-1} + \Phi(h)}\), where
	\[\Phi(h) = p^{h-1} +p^{h-2}+...+p+1.\] This accuracy matches that of previous computations at height \(h=2\).

\section*{Acknowledgments} The author would like to thank Agn\`es Beaudry for meticulously checking multiple drafts of this paper, for many helpful conversations, and for suggesting this project.

	\section{The universal deformation $F(x,y)$}\label{computation of F}
	
		Fix a prime \(p\). In this section we consider the Honda formal group law at height \(h>2\). We write 
		 \[\HH = \HH_h.\] This is the unique $p$-typical formal group law over $\FF_p$ with $p$-series
	\[[p]_{\mathbb{H}}(x)=x^{p^h}.\]
	We note that  $\mathbb{H}$ is also a formal group law over $\mathbb{F}_{p^h}$. 
	
	We begin with some generalities. Let \(\Gamma\) be a formal group law over a perfect field \(k\) of characteristic \(p\). For any complete local ring \(A\) with maximal ideal \(\mathfrak{m}\) and projection \(\pi:A\to A/\mathfrak{m}\), we may define the category of deformations of \(\Gamma\) to \(A\), written \(\text{Def}_\Gamma(A)\). The objects of \(\text{Def}_\Gamma(A)\)
	 are pairs \((H,i)\) where \(H\) is a formal group law over \(A\) and \(i:k \to A/\mathfrak{m}\) is a field inclusion such that
	\[ \pi_*H = i_*\Gamma. \]
	A morphism \((H_1,i_1) \to (H_2,i_2)\) is defined only when \(i_1=i_2\). In this case, it is a \(\star\)-isomorphism \(f:H_1\to H_2\), i.e. \(f\) is an isomorphism such that
	\(f(x) \equiv  x \mod \mathfrak{m}\).
	
	 In this paper, we focus on the case \(\Gamma = \HH\). The ring \(R \cong \mathbb{W} [[u_1,...,u_{h-1}]]\) is complete with respect to its maximal ideal \(I_{h} = (p,u_1,u_2,...,\u)\), and represents the functor \(A \mapsto \text{Def}_\HH(A)\).  We will construct a universal deformation \((F,
	\mathrm{id}) \in \text{Def}_{\HH}(R)\) with the property that if \((H,i)\in \text{Def}_{\HH}(A)\), then there exists a unique continuous \(\WW\)-algebra homomorphism
	\[\iota: R \to A \]
	and a unique \(\star\)-isomorphism \(\iota_*F \overset{\cong}\to H \). 
	
	To define \(F\), we consider the universal
	ring \(V \cong \ZZ_{(p)}[v_1,v_2,...]\) which represents \(p\)-typical formal group laws over \(\ZZ_{(p)}\)-algebras. Let \(G(x,y)\) be the universal \(p\)-typical formal group law over \(V\).  Let
	\[\log_G(x) =  \sum_{i\geq 0}l_i x^{p^i}\]
	denote 	the logarithm of \(G\). Here \(l_0=1\).
	We use the Araki generators \(v_1,v_2,...\) which are defined by \(v_0=p\), and
	\begin{equation}\label{araki}
		p l_n = \sum_{0\leq i \leq n}l_i v_{n-i}^{p^i}.
	\end{equation}

	Summing \eqref{araki} over all \(n\geq 0\) and then applying \(\exp_G\), the inverse of \(\log_G\), to both sides gives
	\[ [p]_G(x)= {\sum_{i\geq 0}}^G v_i x^{p^i}.  \]
	Define \(\phi : V\to R\) by
	\begin{equation}\label{phi}
		\phi(v_i) =\begin{cases} u_i & \text{if} \ 1\leq i \leq h-1 \\
			1 & \text{if} \ i=h \\
			0 & \text{if} \ i>h. 
		\end{cases}
	\end{equation}
We define \[F(x,y) =\phi_*G(x,y).\]
Letting \(L_i = \phi(l_i)\), the formal group law \(F(x,y)\)  is \(p\)-typical with \(p\)-series
	\[[p]_F(x)= px \F u_1x^p \F...\F \u x^{p^{h-1}} \F x^{p^h},\]
	logarithm
	\[\log_F(x) = \sum_{i\geq 0} L_ix^{p^i},\]
	and is a universal deformation of \(\HH\) to \(R\).
	
	From now on we will abbreviate $\log(x) = \log_F(x)$ and will let $\exp(x) = \exp_F(x)$ be the inverse of $\log(x)$ so that 
	\begin{equation}\label{explog}
		F(x,y) = \exp( \log(x) + \log(y)).
		\end{equation}
	
	We note that \eqref{araki} implies, modulo \((u_1,...,u_{h-2})\), that \(L_i =0 \) for \(0\leq i <h-1\) and
	\begin{equation}\label{Ls}
		L_{h-1} = \frac{\u}{p-p^{p^{h-1}}}, \hspace{1cm} L_h = \frac{1}{p-p^{p^h}}.
		\end{equation}
	
We will use \eqref{explog} to prove Theorem \ref{F(x,y) computation}. We first need to calculate \(\exp(x)\), at least up to a sufficiently high power of \(x\). We follow Lader's approach in \cite{lader2013resolution}.

\begin{theorem}\label{exp}
	Assume \(h >2\). Given 
	\[\log(x) = x + L_{h-1}x^{p^{h-1}} + L_hx^{p^h} \mod( x^{p^{h+1}}),  \]
	the inverse series is given by
		\[\exp(x) = x +\sum_{j=0}^{p-1} \frac{(-1)^{j+1}}{j+1} {p^{h-1}(j+1) \choose j }L_{{h-1}}^{j+1}x^{j(p^{h-1}-1)+p^{h-1}} -L_{h}x^{p^h} \ \mod (x^{p^h+1}).\]
	\end{theorem}

	\begin{proof}
		
		For a formal power series		
		\[f(x) = \sum_{i\geq 1} a_ix^i\]
		with inverse
		\[g(x) = \sum_{i\geq 1}b_i x^i,\]		
		the Lagrange inversion formula (see 4.5.12 in
		\cite{morse1954methods}) gives the coefficients of \(g(x)\) as follows:
		\[b_1 = \frac{1}{a_1}\]
	and for \(n>1\)	
		\begin{equation}\label{b_n}
			b_n = \frac{1}{na_1^n}\sum_{c_1,c_2,\ldots}(-1)^{c_1+c_2+\ldots} \frac{n(n+1)\ldots (n-1+c_1+c_2+\ldots )}{c_1!c_2!\ldots} {\left(\frac{a_2}{a_1}\right)}^{c_1}{\left(\frac{a_3} {a_1}\right)}^{c_2}\ldots. 
			\end{equation}
		Where \( c_1,c_2,c_3,... \geq 0\) and
		\[ c_1 + 2c_2 + 3c_3+... = n-1.\]
		In the present case, 
		\[ f(x) = \log(x) = x+L_{{h-1}}x^{p^{h-1}} + L_{h}x^{p^h} \mod (x^{p^{h+1}}). \]
		Hence
		\[ a_i = \begin{cases}
			1 & \text{if} \ i=1 \\
			L_{{h-1}} & \text{if} \ i = p^{h-1} \\
			L_{h} & \text{if} \ i=p^h\\
			0 & \text{otherwise}
		\end{cases}.\]		
		
		Therefore, for fixed \(n>1\), the only nonzero terms in \eqref{b_n} occur when \(c_i=0\) for all \(i\) except \(i=p^{h-1}-1\) or \(i=p^h-1\). Whence \eqref{b_n} reduces to
		\begin{equation}\label{b_n reduced}
			b_n = \frac{1}{n}\sum_{c_1,c_2,\ldots}(-1)^{c_1+c_2+\ldots} \frac{n(n+1)\ldots (n-1+c_1+c_2+\ldots )}{c_1!c_2!\ldots} {{a_{p^{h-1}}}}^{c_{p^{h-1}-1}}{a_{p^{h}} }^{c_{p^h-1}},
			\end{equation}
	where
		\begin{equation}\label{c_i}
			(p^{h-1} -1)c_{p^{h-1}-1} + (p^h-1)c_{p^h-1} = n-1. 
		\end{equation}
		Since we're working modulo \(x^{p^{h}+1}\), we take $1 < n \leq p^h$. 
		
		If $n=p^h$, the only solution to \eqref{c_i} is  $c_{p^h-1} =1$ and $c_{p^{h-1}-1}=0$. If $n <p^h$, it's clear that $c_{p^h-1} =0$. In this case, the only solutions to \eqref{c_i} are $c_{p^{h-1}-1}=i$, when $n = i(p^{h-1}-1)+1$, $i=1,2,...,p$.  We substitute into \eqref{b_n reduced} to arrive at
		\begin{align*}
			b_n &= \begin{cases}
				\frac{(-1)^i}{n} \frac{n(n+1)...(n-1+i)}{i!}L_{h-1}^i & \text{if} \ n=i(p^{h-1}-1)+1\\
				-L_{h} & \text{if} \ n=p^h
			\end{cases}\\
			&= \begin{cases}
				\frac{(-1)^i}{i}{{p^{h-1}i} \choose {i-1}}L_{{h-1}}^i & \text{if} \ n=i(p^{h-1}-1)+1 \\
				-L_{h} & \text{if} \ n=p^h.
			\end{cases}
		\end{align*}

		Therefore,
		\[\exp(x) = x+\left ( \sum_{i=1}^p \frac{(-1)^i}{i}{{p^{h-1}i} \choose {i-1}}L_{{h-1}}^i x^{i(p^{h-1}-1)+1}    \right ) - L_{h}x^{p^h} \ \mod (x^{p^h+1}). \]
		 We reindex $\exp(x)$ by setting $j=i-1$ to complete the proof.	
		\end{proof}		
		
\begin{rem}\label{h>2 is necessary}
	{Since $h>2$, $i=p+1$ does not give a solution to \eqref{c_i}. This is the only part of the proof where we use this assumption. Cf. \cite{lader2013resolution} Lemma 3.1.}
\end{rem}

We now come to the main result of this section. 
	
		\begin{restatable}{theorem}{universaldeformation}
			\label{F(x,y) computation}
			For $h>2$ and $F(x,y)$ as above, modulo $(u_1,...,u_{h-2},(x,y)^{p^{h}+1})$
			\begin{align*}
				F(x,y) &= x +y -\frac{u_{h-1}}{1-p^{p^{h-1}-1}}C_{p^{h-1}}(x,y) - \frac{1}{1-p^{p^h-1}}C_{p^h}(x,y) +\sum_{m=1}^{p-1}\u^{m+1}P_{m}(x,y)
			\end{align*}
			where 
			$$C_{p^n}(x,y) = \frac{1}{p}((x+y)^{p^n} - x^{p^n} -y^{p^n})$$
			and
			\begin{align*}P_{m}(x,y) =  \frac{1}{(p-p^{p^{h-1}})^{m+1}} \sum_{j=0}^{m} \frac{(-1)^{j+1}}{j+1}&{{p^{h-1}(j+1)} \choose j} {{j(p^{h-1}-1)+p^{h-1}} \choose {m-j}}\\
				&(x+y)^{p^{h-1}(j+1)-m}(x^{p^{h-1}} + y^{p^{h-1}})^{m-j}.
			\end{align*}
			
		\end{restatable}		
				
	\begin{proof}			
				
		We will work modulo   $ (u_1,...,u_{h-2}, (x,y)^{p^h+1})$ to compute \(F(x,y)\) as \[F(x,y) = \exp(\log(x) + \log(y)).\]
		We recall that
			\[ \log(x) = x+L_{{h-1}}x^{p^{h-1}} + L_{h}x^{p^h} \mod (x^{p^{h+1}}), \]
		 and by Theorem \ref{exp},
		\[\exp(x) = x +\sum_{j=0}^{p-1} \frac{(-1)^{j+1}}{j+1} {p^{h-1}(j+1) \choose j }L_{{h-1}}^{j+1}x^{j(p^{h-1}-1)+p^{h-1}} -L_{h}x^{p^h} \ \mod (x^{p^h+1}).\]
	We first evaluate the sum indexed by $j$ above  at $\log(x)+\log(y)$ modulo \((x,y)^{p^h+1}\) to obtain
		\begin{align*}
			&\sum_{j=0}^{p-1} \frac{(-1)^{j+1}}{j+1} {p^{h-1}(j+1) \choose j }L_{{h-1}}^{j+1}((x+y)+L_{{h-1}}(x^{p^{h-1}}+y^{p^{h-1}})+L_{h}(x^{p^h}+y^{p^h}))^{j(p^{h-1}-1)+p^{h-1}} \\
			&=\sum_{j=0}^{p-1} \frac{(-1)^{j+1}}{j+1} {p^{h-1}(j+1) \choose j }L_{{h-1}}^{j+1}((x+y)+L_{{h-1}}(x^{p^{h-1}}+y^{p^{h-1}}))^{j(p^{h-1}-1)+p^{h-1}} \\
			&=\sum_{j=0}^{p-1} \frac{(-1)^{j+1}}{j+1} {{p^{h-1}(j+1)} \choose j}L_{{h-1}}^{j+1} \sum_{l=0}^{j(p^{h-1}-1)+p^{h-1}} {{j(p^{h-1}-1)+p^{h-1}} \choose l}\\
			& \hspace{8.5cm} (x+y)^{j(p^{h-1}-1)+p^{h-1}-l}L_{h-1}^l(x^{p^{h-1}}+y^{p^{h-1}})^l\\
			&=\sum_{j=0}^{p-1} \sum_{l=0}^{j(p^{h-1}-1)+p^{h-1}} \frac{(-1)^{j+1}}{j+1}{{p^{h-1}(j+1)} \choose j}{{j(p^{h-1}-1)+p^{h-1}} \choose l} L_{{h-1}}^{j+l+1}\\
			& \hspace{8.5cm} (x+y)^{j(p^{h-1}-1)+p^{h-1}-l}(x^{p^{h-1}}+y^{p^{h-1}})^l. 
		\end{align*}
		For fixed \(j,l\), the terms of the above sum are homogeneous  of degree
		\[j(p^{h-1}-1)+p^{h-1}-l+p^{h-1}l = (j+l)(p^{h-1}-1) +p^{h-1},\]
		and therefore will vanish modulo \( (x,y)^{p^{h}+1}\) if $j+l \geq p$. So we may restrict the upper bound on the inner sum to $l = p-j-1$ to obtain
		\begin{align*}
			&\sum_{j=0}^{p-1} \sum_{l=0}^{p-j-1} \frac{(-1)^{j+1}}{j+1}{{p^{h-1}(j+1)} \choose j}{{j(p^{h-1}-1)+p^{h-1}} \choose l} \\
			&\hspace{8cm} L_{{h-1}}^{j+l+1} (x+y)^{j(p^{h-1}-1)+p^{h-1}-l}(x^{p^{h-1}}+y^{p^{h-1}})^l \\
			&=\sum_{j=0}^{p-1} \sum_{m=j}^{p-1} \frac{(-1)^{j+1}}{j+1}{{p^{h-1}(j+1)} \choose j}{{j(p^{h-1}-1)+p^{h-1}} \choose {m-j}}\\
			&\hspace{8cm} L_{{h-1}}^{m+1} (x+y)^{p^{h-1}(j+1)-m}(x^{p^{h-1}}+y^{p^{h-1}})^{m-j} 
		\end{align*}
		where we have set $m=j+l$ in the second line.
		
		The conditions $0\leq j \leq p-1$ and $ j \leq m \leq p-1$ are equivalent to $0\leq m \leq p-1$ and $0\leq j \leq m$.  Therefore, the middle term of $\exp(\log(x)+\log(y))$ is equal to
		\[
		\sum_{m=0}^{p-1} \sum_{j=0}^{m} \frac{(-1)^{j+1}}{j+1}{{p^{h-1}(j+1)} \choose j}{{j(p^{h-1}-1)+p^{h-1}} \choose {m-j}} L_{{h-1}}^{m+1} (x+y)^{p^{h-1}(j+1)-m}(x^{p^{h-1}}+y^{p^{h-1}})^{m-j}.
		\]
		
		The last term of $\exp(x)$ , $L_{h}x^{p^h}$ evaluated at $\log(x)+\log(y)$ becomes
		$$L_{h}((x+y) + L_{{h-1}}(x^{p^{h-1}} + y^{p^{h-1}}) + L_{h}(x^{p^h} + y^{p^h}))^{p^h} = L_{h}(x+y)^{p^h} \ \mod (x,y)^{p^h+1}.$$

		Finally, we combine the above results and use \ref{Ls} to substitute for \(L_h\) and \(L_{h-1}\) to obtain
		\begin{align*}
			F(x,y) = x&+y + L_{h-1}(x^{p^{h-1}}+y^{p^{h-1}})+ L_h(x^{p^h} + y^{p^h}) \\
			&+ \sum_{m=0}^{p-1}\sum_{j=0}^{m}\frac{(-1)^{j+1}}{j+1}{{p^{h-1}(j+1)}\choose{j}}{{j(p^{h-1}-1)+p^{h-1}}\choose {m-j}}L_{h-1}^{m+1} \\
			&\hspace{5cm} (x+y)^{p^{h-1}(j+1)-m}(x^{p^{h-1}}+y^{p^{h-1}})^{m-j} \\
			& -L_h(x+y)^{p^h} \\
			= x&+y -pL_{h-1}C_{p^{h-1}}(x,y) - pL_hC_{p^h}(x,y) \\
			&+ \sum_{m=1}^{p-1}\sum_{j=0}^{m}\frac{(-1)^{j+1}}{j+1}{{p^{h-1}(j+1)}\choose{j}}{{j(p^{h-1}-1)+p^{h-1}}\choose {m-j}}L_{h-1}^{m+1} \\
			&\hspace{5cm} (x+y)^{p^{h-1}(j+1)-m}(x^{p^{h-1}}+y^{p^{h-1}})^{m-j} \\
			=x& +y -\frac{u_{h-1}}{1-p^{p^{h-1}-1}}C_{p^{h-1}}(x,y) - \frac{1}{1-p^{p^h-1}}C_{p^h}(x,y) \\
			&+\sum_{m=1}^{p-1}\u^{m+1}P_{m}(x,y).
	\qedhere	\end{align*}
	\end{proof}

	\section{The action of \(\SSS_h\) on \(R\)}\label{The Action}
	We now turn to the action of the automorphism group \(\SSS=\SSS_h\) of the height \(h>2\) Honda formal group law \(\HH\) on the associated Lubin-Tate deformation ring \(R\) modulo the invariant ideal \[I_{h-1}=(p,u_1,...,u_{h-2}).\]
	Good references for this material are Appendix A2 of \cite{ravenel2003complex} and section 4 of \cite{henn2013homotopy}.

We recall that \(\SSS\) is the subgroup of units in
	\[ \WW\langle S \rangle / (S^h=p, Sa=a^{\sigma}S )\]
where \(\WW = W(\FF_{p^h})\) are the Witt vectors, \(a\in \WW\) and \(\sigma\) is a lift of the Frobenius. Therefore 
any \(g\in \SSS\) can be expressed as
\[g = \sum_{i\geq 0}g_i S^i.\] 

We now describe the action of \(\mathbb{S}\) on \(R\). We may lift  \(g\in \SSS\) to an element \(\hat{g}(x) \in \WW[[x]]\) and define a formal group law over \(R\) by
	\[\hat{F}(x,y)= \hat{g}^{-1}(F(\hat{g}(x),\hat{g}(y))).\]
	We note that \((\hat{F},\mathrm{id})\) is a deformation of \(\HH\) to \(R\), and therefore determines a unique homomorphism \(g_*:R\to R\) and a unique \(\star\)-isomorphism from \(g_*F\) to \(\hat{F}\). The composition
	\[h_g: g_*F \to \hat{F} \overset{\hat{g}}\to F \]
	 is independent of the choice of lift \(\hat{g}\). Since \(h_g\) is an isomorphism of \(p\)-typical formal group laws, it can be written as
	\[h_g(x) = {\sum_{i\geq 0}}^F t_i(g)x^{p^i}\] for unique continuous functions 
	$t_i: \mathbb{S} \to R$. We will abbreviate $t_i = t_i(g)$. The $t_i$ have the property that $t_i(g) = g_i$ modulo $(p,u_1,...,\u)$, and \(t_0\) is a unit.  Since \(h_g\) is a morphism of formal group laws,
	\begin{equation}\label{recursion}
		h_g([p]_{g_*F}(x)) = [p]_{F}(h_g(x)).
	\end{equation}

	The following easy theorem describes the action of \(g\in \mathbb{S}\) in terms of the function \(t_0\).
	\begin{restatable}{theorem}{action}\label{action}
		Let $g\in \mathbb{S}$ as above, and let \(h\geq 2\). Then
		\[g_*(\u) = \u t_0^{p^{h-1}-1} \ \mod I_{h-1}.\]
	\end{restatable}

	\begin{proof} The case \(h=2\) is proved in \cite{lader2013resolution}, so we assume \(h>2\).
		We will expand both sides of \eqref{recursion}, working modulo \(x^{p^{h-1}+1}\). We recall that
		\[[p]_F(x) =\u x^{p^{h-1}} \mod (I_{h-1},x^{p^{h-1}+1}).\]

The left side of \eqref{recursion} is		
		\begin{dmath*}
			h_g([p]_{g_*F}(x)) = t_0 g_*(\u)x^{p^{h-1}}
		\end{dmath*}
and the right side is
\begin{dmath*}
	[p]_F(h_g(x))
	= \u(t_0x)^{p^{h-1}}.
	\end{dmath*}
		Comparing the coefficient of $x^{p^{h-1}}$ completes the proof.
	\end{proof}

	The remainder of this section is devoted to computing explicit, though recursive, formulas for the \(t_i\)'s, \(0\leq i \leq h-1\).
	
	We begin by establishing several lemmas that will be invoked frequently in the computations later in this section.
	
	\begin{lemma}\label{distributor}
		Let \(\Gamma(s,t)\) be a formal group law over a ring \(\Sigma\) in which \(p=0\), let \(d\) be a positive integer, and let \(A,B,C \in \Sigma[[x]]\). If \[(A \underset{\Gamma}+ B )^{p^l} = (A+B)^{p^l} \ \mod (x^d),\]
		then
		
		\[(A \underset{\Gamma}+ B \underset{\Gamma}+ C)^{p^l} = ( (A+B) \underset{\Gamma}+ C)^{p^l} \ \mod (x^d).\]
		
	\end{lemma}
	\begin{proof}
		
		The formal group law \(\Gamma\) can be expressed generically as
		\[ \Gamma(s,t) = s+t + \sum_{i,j >0} a_{ij}s^i t^j.\]
		
		By associativity and the fact that \(p=0\), 
		\begin{align*}
			(A \underset{\Gamma}+ B \underset{\Gamma}+ C)^{p^l} &= (A \underset{\Gamma}+ B)^{p^l} + C^{p^l} + \sum_{i,j>0}a_{ij}^{p^l}(A \underset{\Gamma}+ B)^{p^l} C^{p^l}\\
			&\equiv (A + B)^{p^l} + C^{p^l} + \sum_{i,j>0}a_{ij}^{p^l}(A + B)^{p^l} C^{p^l}\\
			&\equiv ((A+B) \underset{\Gamma}+ C)^{p^l}.
\qedhere		\end{align*}
	\end{proof}

	The following technical lemma will be invoked frequently, and is of central importance to the rest of the computations in this section. 
	
	\begin{lemma}\label{slayer}  Let \(R\) be the Lubin-Tate ring, and \(A,B \in R[[x]] \). Assume that \(x^a\) is the highest power of \(x\) dividing \(A\), \(x^b\) the highest power of \(x\) dividing \(B\), and \(a\leq b \leq a p^{h-1}\). For \(F(x,y)\) as in Theorem \ref{F(x,y) computation} and \(l\geq 0\), if 
		
		\[(a(p^{h-1}-1) +b)p^l \geq d,  \]
	then	
		\[(A \F B)^{p^l} = (A+B)^{p^l} \mod (I_{h-1},x^d).\]
		
	\end{lemma}
	
	\begin{proof} By Theorem \ref{F(x,y) computation}, and the fact that we are working mod \((p)\), 
		\begin{dmath*}
			(A\F B)^{p^l} \equiv A^{p^l} +B^{p^l} +\left(-\frac{u_{h-1}}{1-p^{p^{h-1}-1}}C_{p^{h-1}}(A,B)\right)^{p^l} +\left(- \frac{1}{1-p^{p^h-1}}C_{p^h}(A,B)\right)^{p^l} +\left(\sum_{m=1}^{p-1}\u^{m+1}P_{m}(A,B)\right)^{p^l}.
		\end{dmath*}
	
	We will analyze each of the terms in the above expansion of \(F(x,y)\), starting with \(C_{p^n}(A,B)\).
	
			By definition,
		\[C_{p^n}(A,B) = \frac{1}{p}((A+B)^{p^n}-A^{p^n} - B^{p^n}).\] 
		Therefore, the highest power of \(x\) that divides	\(C_{p^{n}}(A,B)^{p^l}\) is 
		\[(a(p^{n}-1)  +b)p^l. \]	
		Taking \(n= p^{h-1}\), the hypotheses imply that both \( C_{p^{h-1}}(A,B)^{p^l}\) and \(C_{p^h}(A,B)^{p^l}\) vanish modulo \((x^d)\).

		Next, we recall that
		\begin{align*}P_{m}(A,B) =  \frac{1}{(p-p^{p^{h-1}})^{m+1}} \sum_{j=0}^{m} \frac{(-1)^{j+1}}{j+1}&{{p^{h-1}(j+1)} \choose j} {{j(p^{h-1}-1)+p^{h-1}} \choose {m-j}}\\
			&(A+B)^{p^{h-1}(j+1)-m}(A^{p^{h-1}} + B^{p^{h-1}})^{m-j}
		\end{align*}
		with \(m \geq1\). The highest power of \(x\)  dividing \(P_m(A,B)^{p^l}\) is
		
		\begin{dmath*} (a(p^{h-1}(j+1) -m) + a p^{h-1}(m-j))p^l =  (ap^{h-1} + m(ap^{h-1}-a))p^l
			\geq (2ap^{h-1} -a)p^l. 
		\end{dmath*}
		
		Finally, we note that
		\[	(a(p^{h-1}-1) + b)p^l \leq (2ap^{h-1} -a)p^l\] 
		if, and only if
		\[b\leq ap^{h-1}.\]
		This completes the proof.
		\end{proof}
	
	\begin{note}\label{gslayer} Since the monomial degrees of \(g_*F(x,y)\) are the same as those of \(F(x,y)\), under the hypotheses of Lemma \ref{slayer},
		\[(A \gF B)^{p^l} = (A+B)^{p^l} \mod (I_{h-1},x^d)   \]
		for any \( g \in \SSS\).
	\end{note}
	
We are now prepared for the first of two main computations in this section. The proof is inspired by \cite{lader2013resolution}.

	\begin{restatable}{theorem}{tk}\label{tk} For \(h\geq 2\), \(g\in \SSS\) as above, and \(0\leq k <h-1\), 
	\[t_k = t_k^{p^h}+\u t_{k+1}^{p^{h-1}}-\u^{p^{k+1}}t_{k+1}t_0^{p^{k+1}(p^{h-1}-1)} \mod I_{h-1}.\]
\end{restatable}
	\begin{proof} The case \(h=2\) was computed by Lader, so we assume \(h>2\).
		We fix an integer \(k\) with \(0\leq k < h-1 \).	We will compute the coefficient of \(x^{p^{h+k}}\) on both sides of \eqref{recursion}. Throughout the proof, all equalities are understood to be modulo (\(I_{h-1},x^{p^{h+k}+1})\). 
		
		We begin with the left-hand side of \eqref{recursion}. We apply Lemma \ref{distributor} and Lemma \ref{slayer} with \(d=p^{h+k}+1\) to obtain
		\begin{dmath*}
			h_g([p]_{g_*F}(x)) = t_0(g_*(\u)x^{p^{h-1}} \gF x^{p^h})
			\F t_1(g_*(\u)x^{p^{h-1}} \gF x^{p^h})^p
			\F \cdots \F
			t_k(g_*(\u)x^{p^{h-1}} \gF x^{p^h})^{p^k} 
			\F t_{k+1}(g_*(\u)x^{p^{h-1}} \gF x^{p^h})^{p^{k+1}}
			 = t_0(g_*(\u)x^{p^{h-1}} + x^{p^h})
			\F t_1(g_*(\u)x^{p^{h-1}} + x^{p^h})^p
			\F \cdots \F
			t_k(g_*(\u)x^{p^{h-1}} + x^{p^h})^{p^k} 
			\F t_{k+1}(g_*(\u)x^{p^{h-1}} + x^{p^h})^{p^{k+1}}
			= t_0(g_*(\u)x^{p^{h-1}} + x^{p^h})
			+ t_1(g_*(\u)x^{p^{h-1}} + x^{p^h})^p
			+ \cdots +
			t_k(g_*(\u)x^{p^{h-1}} + x^{p^h})^{p^k} 
			+ t_{k+1}(g_*(\u)x^{p^{h-1}} + x^{p^h})^{p^{k+1}}.
		\end{dmath*} 
		By Theorem \ref{action}, \(g_*(\u) = \u t_0^{p^{h-1}-1}\). Hence the coefficient of \(x^{p^{h+k}}\) on the left-hand side of \eqref{recursion} is
		\[t_k+ t_{k+1}\u^{p^{k+1}}t_0^{p^{k+1}(p^{h-1}-1)}.\]
		
	We now turn to the right-hand side of \eqref{recursion}. We simplify using Lemmas \ref{slayer} and \ref{distributor} to obtain
		\begin{dmath*}
			[p]_F(h_g(x)) = \u (t_0x \F t_1x^p \F \cdots \F t_kx^{p^k} \F t_{k+1}x^{p^{k+1}})^{p^{h-1}} \F (t_0x \F t_1x^p \F \cdots \F t_{k-1}x^{p^{k-1}} \F t_{k}x^{p^{k}})^{p^{h}}
	= \u (t_0x + t_1x^p + \cdots + t_kx^{p^k} + t_{k+1}x^{p^{k+1}})^{p^{h-1}} \F (t_0x + t_1x^p + \cdots + t_{k-1}x^{p^{k-1}} + t_{k}x^{p^{k}})^{p^{h}}
			= \u (t_0x + t_1x^p + \cdots + t_kx^{p^k} + t_{k+1}x^{p^{k+1}})^{p^{h-1}} + (t_0x + t_1x^p + \cdots + t_{k-1}x^{p^{k-1}} + t_{k}x^{p^{k}})^{p^{h}}.
		\end{dmath*}
		
		The coefficient of \(x^{p^{h+k}}\) is
		\[\u t_{k+1}^{p^{h-1}} + t_k^{p^h}.\]
		
		We equate coefficients to complete the proof.	
	\end{proof}

The case \(k=h-1\) is significantly more difficult, and is the last result in this section. The following lemma will be used frequently in that computation.

\begin{lemma}\label{Cpn} For any \(n\geq 1\)
\[C_{p^n}(x,y) = C_p(x^{p^{n-1}},y^{p^{n-1}}) \mod (p).\]	
	\end{lemma}
\begin{proof}
The claim is trivial for \(n=1\), so we assume \(n \geq 2\).	
We first observe that	
	\[(x+y)^{p^{n-1}} = x^{p^{n-1}} + pX + y^{p^{n-1}}\]
	where \(X\) is the remainder modulo \((p)\). Then
	\begin{dmath*}
	C_{p^n}(x,y) = \frac{1}{p} \left(  (x+y)^{p^n} - x^{p^n} - y^{p^n}   \right)
	= 	\frac{1}{p} \left(  \left(  x^{p^{n-1}} + pX + y^{p^{n-1}}  \right)^{p} - x^{p^n} - y^{p^n}   \right)
	= 	\frac{1}{p} \left( \sum_{i=0}^p {p \choose i} (pX)^i (x^{p^{n-1}} + y^{p^{n-1}})^{p-i}    - x^{p^n} - y^{p^n}   \right)
	=  \sum_{i=1}^p \frac{1}{p}{p \choose i} (pX)^i (x^{p^{n-1}} + y^{p^{n-1}})^{p-i}  + \frac{1}{p}\left((x^{p^{n-1}} + y^{p^{n-1}})^p    - x^{p^n} - y^{p^n}   \right).
		\end{dmath*}
	We evaluate modulo \((p)\) to arrive at 
	\begin{align*}
			C_{p^n}(x,y) &= \frac{1}{p}\left((x^{p^{n-1}} + y^{p^{n-1}})^p    - x^{p^n} - y^{p^n}   \right) \\
			&= C_p(x^{p^{n-1}},y^{p^{n-1}} ).
	 \qedhere \end{align*} 
	\end{proof}

We have split the proof of Theorem \ref{th} into the following two lemmas.

\begin{lemma}\label{rhs} Assume \(h>2\). Modulo \((p,u_1,...,u_{h-2},\u^{p^{h-1}-1})\), the coefficient of \(x^{p^{2h-1}}\) on the right-hand side of \eqref{recursion} is equal to
	\begin{dmath*}t_{h-1}^{p^{h}} + \u t_h^{p^{h-1}} - \sum_{j=1}^{p-1}\frac{1}{p}{p \choose j}\u^{jp^{h-2}+1}t_1^{jp^{2h-3}}t_0^{(p-j)p^{2h-2}}.
		\end{dmath*}
	\end{lemma}

	\begin{proof} Throughout this proof, all equalities are understood to be modulo \((x^{p^{2h-1}+1})\).
		
	The right-hand side is given by
		\begin{dmath*} 
			[p]_{F}(h_g(x)) = \u(h_g(x))^{p^{h-1}} \F (h_g(x))^{p^h}
			=\u(t_0x \F t_1x^p \F ... \F t_{h-1}x^{p^{h-1}} \F t_hx^{p^h})^{p^{h-1}} \F (t_0x \F t_1x^p \F ... \F t_{h-2}x^{p^{h-2}} \F t_{h-1}x^{p^{h-1}})^{p^h}.
		\end{dmath*}
		
		Let
		\[A= \u(t_0x \F t_1x^p \F ... \F t_{h-1}x^{p^{h-1}} \F t_hx^{p^h})^{p^{h-1}}  \]
		\[B= (t_0x \F t_1x^p \F ... \F t_{h-2}x^{p^{h-2}} \F t_{h-1}x^{p^{h-1}})^{p^h}.\]
		
By Lemmas \ref{distributor} and \ref{slayer},
		
		\[A= \u(t_0x \F [t_1x^p + ... + t_{h-1}x^{p^{h-1}} + t_hx^{p^h}])^{p^{h-1}} \mod (x^{p^{2h-1}+1}) \]
		and
		\[B= (t_0x + t_1x^p + ... + t_{h-2}x^{p^{h-2}} + t_{h-1}x^{p^{h-1}})^{p^h} \mod (x^{p^{2h-1}+1}).\]
		
		We claim that, modulo \((x^{p^{2h-1}+1})\),
		\[A\F B = A+B - \u C_{p^{h-1}}(A,B). \]
		
	We first note that the smallest degree term in \(C_{p^{h}}(A,B)\)  has degree
	\[p^{h-1}(p^h-1)+p^h > p^{2h-1}.\]
		
	Thus, to prove the claim, it suffices to show \[P_m(A,B) \equiv 0.\] Ignoring all integer coefficients, we need to analyze
		\[\sum_{m=1}^{p-1} \u^{m+1} \sum_{j=0}^{m}(A+B)^{p^{h-1}(j+1)-m} (A^{p^{h-1}} + B^{p^{h-1}})^{m-j}.\]
		
		The smallest degree term appearing in \(B^{p^{h-1}}\) is \(p^{2h-1}\). Since the exponent on \((A+B)\) is positive,
		all terms involving \(B^{p^{h-1}}\) vanish modulo \((x^{p^{2h-1}+1})\). We are left with
		\begin{dmath*}
			\sum_{m=1}^{p-1} \u^{m+1} \sum_{j=0}^{m}(A+B)^{p^{h-1}(j+1)-m} A^{p^{h-1}(m-j)}
			=	\sum_{m=1}^{p-1} \u^{m+1} \sum_{j=0}^{m}\sum_{i=0}^{p^{h-1}(j+1)-m}
			{{p^{h-1}(j+1)-m} \choose {i}} A^{p^{h-1}(m-j)+i} B^{p^{h-1}(j+1)-m-i}.
		\end{dmath*}

		Since \(\u\) divides \(A\), the power of \(\u\) appearing in this sum is
		\begin{dmath*}
			m+1+p^{h-1}(m-j)+i.
		\end{dmath*}
		
		If \(m-j \geq 1\), then all terms vanish modulo \((u^{p^{h-1}+1})\). So we assume \(m=j\). The corresponding \(P_m\) term is
		\begin{dmath*}	\sum_{m=1}^{p-1} \u^{m+1} \sum_{i=0}^{p^{h-1}(m+1)-m}
			{{p^{h-1}(m+1)-m} \choose {i}} A^{i} B^{p^{h-1}(m+1)-m-i}.
		\end{dmath*}

	Since the exponent of \(\u\) is equal to \(m+i\),  the only nonzero terms occur only when
		\[ m+i < p^{h-1}.\]
		
		The corresponding powers of \(x\) have degree at least
		
		\begin{dmath*}
			p^{h-1}i + p^{h}(p^{h-1}(m+1)-m-i) > 	p^{h-1}i + p^{h}(p^{h-1}(m+1)-p^{h-1})
			\geq p^{2h-1}.
		\end{dmath*}

		Therefore the \(P_m(A,B)\) term of \(A\F B\) vanishes. This proves the claim.
		
		Next, we analyze the \(C_{p^{h-1}}(A,B)\) term. By Lemma \ref{Cpn},
	\begin{dmath*}
		C_{p^{h-1}}(A,B)  =\frac{1}{p} \left((A^{p^{h-2}}+B^{p^{h-2}})^{p}-A^{p^{h-1}} - B^{p^{h-1}}\right)
		=\sum_{i=1}^{p-1}\frac{1}{p}{p \choose i}A^{p^{h-2}(p-i)}B^{p^{h-2}i}.
		\end{dmath*}
		We note that
		\begin{dmath*}
			A^{p^{h-2}} =  \u^{p^{h-2}}(t_0x \F [t_1x^p + ... + t_{h-1}x^{p^{h-1}} + t_hx^{p^h}])^{p^{2h-3}} 
			=  \u^{p^{h-2}}(t_0x \F [t_1x^p + t_2x^{p^2}])^{p^{2h-3}}  \mod (x^{p^{2h-1}+1}).
		\end{dmath*}
		Any term involving the monomial \(t_2x^{p^2}\) will vanish in \(C_{p^{h-1}}(A,B)\) modulo \((x^{p^{2h-1}+1})\), since \(B\) has nonzero exponent.
		Therefore we may replace \(A^{p^{h-2}}\) with 
		\[ \u^{p^{h-2}}(t_0x \F t_1x^p)^{p^{2h-3}}. \]
	 Furthermore, since \(h\geq 3\), Lemma \ref{slayer} implies
		\[(t_0x \F t_1x^p)^{p^{2h-3}} = (t_0x + t_1x^p)^{p^{2h-3}} \ \mod (x^{p^{2h-1}+1}). \]

		We are still studying the \(	C_{p^{h-1}}(A,B)\) term, and we will now replace \(B\) with a simpler expression. We note that
		\begin{dmath*}
			B^{p^{h-2}} = (t_0x + t_1x^p + ... + t_{h-2}x^{p^{h-2}} + t_{h-1}x^{p^{h-1}})^{p^{2h-2}}
			=(t_0x+t_1x^p)^{p^{2h-2}} \mod (x^{p^{2h-1}+1}).
		\end{dmath*}
		Since the monomials involving \(A^{p^{h-2}}\) in the 	\(C_{p^{h-1}}(A,B)\) have nonzero degree, all monomials involving the \(t_1x^p\) term in \(B\) will vanish modulo \((x^{p^{2h-1}+1})\). Whence \(B\) 	may be replaced with
		\[t_0^{p^{2h-2}}x^{p^{2h-2}}.\]
		
		In summary, we have shown that, modulo \((x^{p^{2h-1}+1})\),
		\begin{dmath*}
			C_{p^{h-1}}(A,B) = \sum_{i=1}^{p-1}\frac{1}{p}{p \choose i}(\u^{p^{h-2}}(t_0x + t_1x^p)^{p^{2h-3}})^{p-i}(t_0^{p^{2h-2}}x^{p^{2h-2}})^{i}.
		\end{dmath*}

		The largest power of \(x\) in this sum is
		
		\[p^{2h-2}(p-i) + ip^{2h-2} =  p^{2h-1}.\]
		
		Therefore, the coefficient of \(x^{p^{2h-1}}\) from the \(C_{p^{h-1}}(A,B)\) term is
		\begin{equation}\label{C part}
			 -{\u}\sum_{i=1}^{p-1}\frac{1}{p}{p \choose i}(\u^{p^{h-2}}t_1^{p^{2h-3}})^{p-i}(t_0^{p^{2h-2}})^{i}
			= - \sum_{j=1}^{p-1}\frac{1}{p}{p \choose j}\u^{jp^{h-2}+1}t_1^{jp^{2h-3}}t_0^{(p-j)p^{2h-2}}.
\end{equation}
We have reindexed by setting \(j = p-i\).

		We are left to analyze the \(A+B\) term. This term is given by
		\begin{dmath*}
			A+B = \u(t_0x \F [t_1x^p + ... + t_{h-1}x^{p^{h-1}} + t_hx^{p^h}])^{p^{h-1}} + 
			(t_0x + t_1x^p + ... + t_{h-2}x^{p^{h-2}} + t_{h-1}x^{p^{h-1}})^{p^h}.
		\end{dmath*}
		
		The coefficient of \(x^{p^{2h-1}}\) in \(B\) is  \(t_{h-1}^{p^h}\), so the last step is to analyze \[A=\u(t_0x \F [t_1x^p + ... + t_{h-1}x^{p^{h-1}} + t_hx^{p^h}])^{p^{h-1}}.\]
		We let 
		\[A_0 =t_1x^p + ... + t_{h-1}x^{p^{h-1}} + t_hx^{p^h}. \]
		
		First, \(P_m(t_0x,A_0)^{p^{h-1}}=0\) modulo \((\u^{p^{h-1}+1})\) since it is divisible by \(\u^{p^{h-1}(m+1)}\), and \(m\geq 1\). Monomials in \(A\) arising from the \(C_{p^{h}}(t_0x,A_0)\) term have degree at least 
		\[p^{h-1} (p^{h}-1+p) > p^{2h-1},\] so the \(C_{p^{h}}(t_0x,A_0)\) term vanishes modulo \((x^{p^{2h-1}+1})\).

		We are left to compute the \(C_{p^{h-1}}(t_0x,A_0)\) term. By Lemma \ref{Cpn},
		\begin{dmath*}
			C_{p^{h-1}}(t_0x,A_0)^{p^{h-1}} 
			=\left[ \sum_{i=1}^{p-1} \frac{1}{p} {p \choose i} t_0^{ip^{h-2}}x^{ip^{h-2}} A_0^{p^{h-2}(p-i)} \right]^{p^{h-1}}
			= \sum_{i=1}^{p-1} \left[\frac{1}{p} {p \choose i}\right]^{p^{h-1}} t_0^{ip^{2h-3}}x^{ip^{2h-3}} A_0^{p^{2h-3}(p-i)}. 
		\end{dmath*}
		
		We note that modulo \((x^{p^{2h-1}+1})\),
		\[A_0^{p^{2h-3}} =t_1^{p^{2h-3}}x^{p^{2h-2}} + t_2^{p^{2h-3}}x^{p^{2h-1}}.   \]
		
		Monomials involving the \(t_2^{p^{2h-3}}x^{p^{2h-1}}\) term from \(A_0\) will vanish modulo \((x^{p^{2h-1}+1})\) in the \\ \(C_{p^{h-1}}(t_0x,A_0)^{p^{h-1}}\) term. Therefore
		\begin{dmath*}
			C_{p^{h-1}}(t_0x,A_0)^{p^{h-1}} =  \sum_{i=1}^{p-1} \left[ \frac{1}{p} {p \choose i}\right]^{p^{h-1}} t_0^{ip^{2h-3}}x^{ip^{2h-3}} (t_1^{p^{2h-3}}x^{p^{2h-2}})^{p-i} .
		\end{dmath*}

		The degree of \(x\) is
		\[ip^{2h-3}+p^{2h-2}(p-i) = p^{2h-1} - i(p^{2h-2}-p^{2h-3}) < p^{2h-1}.\]
		So the \(C_{p^{h-1}}\) term does not contribute to the coefficient of \(x^{p^{2h-1}}\).
		
		Therefore the coefficient of \(x^{p^{2h-1}}\) in \(A\)
		is the coefficient in
		\[\u (t_0x + (t_1x^p + ... + t_hx^{p^{h}}))^{p^{h-1}},\]
		which is 
	\[	\u t_h^{p^{h-1}}.
			\]
		
		The coefficient of \(x^{p^{2h-1}}\) in the linear term \(A+B\) is thus
		
		\begin{equation}\label{linear part}
			t_{h-1}^{p^{h}} + \u t_h^{p^{h-1}}. 
			\end{equation}
		
		We combine equations \eqref{C part} and \eqref{linear part} to arrive at the coefficient of \(x^{p^{2h-1}}\) on the right-hand side of \eqref{recursion}:
		\[- \sum_{j=1}^{p-1}\frac{1}{p}{p \choose j}\u^{jp^{h-2}+1}t_1^{jp^{2h-3}}t_0^{(p-j)p^{2h-2}} +t_{h-1}^{p^{h}} + \u t_h^{p^{h-1}}. \qedhere
		\] 
\end{proof}

\begin{lemma}\label{lhs} Assume \(h>2\). Modulo \((I_{h-1},\u^{p^{h-1}-1})\), the coefficient of \(x^{p^{2h-1}}\) on the left-hand side of \eqref{recursion} is equal to
	\(t_{h-1}. \)
	\end{lemma}

\begin{proof} Throughout this proof, all equalities are understood to be modulo \((x^{p^{2h-1}+1})\).
	
		The left-hand side of \eqref{recursion} is given by
		\begin{dmath*}
			h_g([p]_{g_*F}(x)) = t_0(g_*(\u)x^{p^{h-1}} \gF x^{p^h}) \F  t_1(g_*(\u)x^{p^{h-1}} \gF x^{p^h})^p \F ... \F  t_{h-1}(g_*(\u)x^{p^{h-1}} \gF x^{p^h})^{p^{h-1}} \F  t_h(g_*(\u)x^{p^{h-1}})^{p^h}
			\overset{(!)}= t_0(g_*(\u)x^{p^{h-1}} \gF x^{p^h})  \F      \left [t_1(g_*(\u)x^{p^{h-1}} + x^{p^h})^p \F ... \F  t_{h-1}(g_*(\u)x^{p^{h-1}} + x^{p^h})^{p^{h-1}} \F  t_h(g_*(\u)x^{p^{h-1}})^{p^h}\right]
			\overset{(!!)}= t_0(g_*(\u)x^{p^{h-1}} \gF x^{p^h})  \F    \left[  t_1(g_*(\u)x^{p^{h-1}} + x^{p^h})^p + ... +  t_{h-1}(g_*(\u)x^{p^{h-1}} + x^{p^h})^{p^{h-1}} +  t_h(g_*(\u)x^{p^{h-1}})^{p^h}\right].
		\end{dmath*}

		Here we use Lemma \ref{slayer}. The relevant inequalities are
		
		\begin{enumerate}
			\item[(!)] 	\(p^k(p^{h-1}(p^{h-1}-1) + p^h) > p^{2h-1} \ \text{and} \ p^h \leq p^{h-1}p^{h-1} \ \text{if, and only if} \  
			k\geq 1	\)
			\item[(!!)] \(p^h(p^{h-1}-1) + p^{h-1+k} > p^{2h-1} \ \text{and} \ p^{2h-1}\leq p^{h-1}p^h \ \text{if, and only if} \ k\geq 2\).
		\end{enumerate}
		
		We need to compute
			\[h_g([p]_{g_*F}(x)) = A \F B\]
		where
		\[A= t_0(g_*(\u)x^{p^{h-1}} \gF x^{p^h}) \]
		\begin{dmath*}
			B= t_1(g_*(\u)x^{p^{h-1}} + x^{p^h})^p + ... +  t_{h-1}(g_*(\u)x^{p^{h-1}} + x^{p^h})^{p^{h-1}} +  t_h(g_*(\u)x^{p^{h-1}})^{p^h} 
			= t_1(g_*(\u)x^{p^{h-1}} + x^{p^h})^p + ... +  t_{h-1}(g_*(\u)x^{p^{h-1}} + x^{p^h})^{p^{h-1}} \  \mod \ (u^{p^{h-1}+1}).
		\end{dmath*}
	We obtain the last equality since \(g_*(\u)\) is divisible by \(\u\) by Theorem \ref{action}.

Before analyzing \(A \F B\), we collect the terms in \(A\) and \(B\) that are divisible by \(\u\). We'll begin with \(A\). We use Theorem \ref{F(x,y) computation} and Theorem \ref{action} to obtain
\begin{dmath*}
	A =t_0\left(\u t_0^{p^{h-1}-1}x^{p^{h-1}} + x^{p^h} -\frac{\u t_0^{p^{h-1}-1}}{1-p^{p^{h-1}-1}}C_{p^{h-1}}(g_*(\u)x^{p^{h-1}}, x^{p^h}) - \frac{1}{1-p^{p^{h}-1}}C_{p^h}(g_*(\u)x^{p^{h-1}}, x^{p^h})\\
	+\sum_{m=1}^{p-1}(\u t_0^{p^{h-1}-1})^{m+1}P_{m}(g_*(\u)x^{p^{h-1}}, x^{p^h})\right).
\end{dmath*}

Since \(C_{p^h}(g_*(\u)x^{p^{h-1}},x^{p^h}) =0 \mod (x^{p^{2h-1}+1})\), we let

\begin{dmath*}
	A_0 = \left( t_0^{p^{h-1}}x^{p^{h-1}} -\frac{t_0^{p^{h-1}}}{1-p^{p^{h-1}-1}}C_{p^{h-1}}(g_*(\u)x^{p^{h-1}}, x^{p^h})  \\
	+\sum_{m=1}^{p-1}t_0^{(p^{h-1}-1)(m+1)+1}\u^mP_{m}(g_*(\u)x^{p^{h-1}}, x^{p^h})\right),
\end{dmath*}
so that
\begin{equation}\label{A}
	A = \u A_0 + t_0x^{p^h}.
	\end{equation}

Next, we expand \(B\) as
\begin{dmath*}
	B= t_1(g_*(\u)x^{p^{h-1}} + x^{p^h})^p + ... +  t_{h-1}(g_*(\u)x^{p^{h-1}} + x^{p^h})^{p^{h-1}}  
	= t_1(\u t_0^{p^{h-1}-1}x^{p^{h-1}} + x^{p^h})^p + ... +  t_{h-1}(\u t_0^{p^{h-1}-1}x^{p^{h-1}} + x^{p^h})^{p^{h-1}} 
	=\u^p \left(t_1t_0^{p^{h}-p}x^{p^h} + ... +t_{h-1}\u^{p^{h-2}}t_0^{p^{2h-2}-p^{h-1}}x^{p^{2h-2}} \right) + \left(t_1x^{p^{h+1}}  + ... + t_{h-1}x^{p^{2h-1}}     \right).
\end{dmath*}

We set
\[ B_0 =t_1t_0^{p^{h}-p}x^{p^h}+ ... +t_{h-1}\u^{p^{h-2}}t_0^{p^{2h-2}-p^{h-1}}x^{p^{2h-2}}  \]
and
\[ B_1 =t_1x^{p^{h+1}}  + ... + t_{h-1}x^{p^{2h-1}}    \]	
so that
\begin{equation}\label{B}
	B=\u^p B_0 + B_1.
	\end{equation}

		We consider first the \(P_m(A,B)\) term in \(A\F B\). Ignoring all integer coefficients, we consider
		\begin{equation}\label{P(A,B)}
			\sum_{m=1}^{p-1}\u^{m+1}\sum_{j=0}^m (A+B)^{p^{h-1}(j+1)-m}(A^{p^{h-1}} + B^{p^{h-1}})^{m-j}.
			\end{equation}
		
		Let us first consider \eqref{P(A,B)} in the case \(m-j \neq 0\). We claim that, in this case,
		 \[P_m(A,B) \equiv 0.\]

		We first observe that
		\begin{dmath*}
			B^{p^{h-1}} = \u^{p^h}B_0^{p^{h-1}}+B_1^{p^{h-1}}
			 =0 \mod (\u^{p^{h-1}+1}, x^{p^{2h-1}+1}).
			\end{dmath*}	
	 Moreover, 
		\begin{dmath*}\u^{m+1}A^{p^{h-1}(m-j)} = \u^{m+1}t_0^{p^{h-1}(m-j)}x^{p^{2h-1}(m-j)} \mod (u^{p^{h-1}+1}).
			\end{dmath*}
	Since the exponent \(p^{h-1}(j+1)-m\) on \(A+B\) is nonzero, and  \(x\) divides \(A+B\), all terms involving  \(\u^{m+1}A^{p^{h-1}(m-j)}\)  vanish modulo \((x^{p^{2h-1}+1})\). This proves the claim.
		
		Therefore we need only consider the case \(m=j\). We set \(m=j\) in \eqref{P(A,B)} to obtain
		\begin{dmath*}
		 \sum_{m=1}^{p-1}\u^{m+1} (A+B)^{p^{h-1}(m+1)-m}=
			\sum_{m=1}^{p-1}\u^{m+1}\sum_{i=0}^{p^{h-1}(m+1)-m} {{p^{h-1}(m+1)-m} \choose {i} }(\u A_0 +t_0x^{p^h})^iB^{p^{h-1}(m+1)-m -i}
			=\sum_{m=1}^{p-1}\sum_{i=0}^{p^{h-1}(m+1)-m}\sum_{j=0}^i  {{p^{h-1}(m+1)-m} \choose {i} }{{i} \choose {j}}\u^{m+1+j} A_0^j t_0^{i-j}x^{p^h(i-j)}B^{p^{h-1}(m+1)-m -i}
		\end{dmath*}

		We remark that this sum has nonvanishing terms modulo \((\u^{p^{h-1}+1})\) only when \[j<p^{h-1} - m . \]
		
		Since \(x^{p^{h-1}}\) divides \(A_0\) and \(x^{p^h}\) divides \(B\), the following power of \(x\) divides the sum

		\begin{dmath*}
			p^{h-1}j + p^{h}(i-j + p^{h-1}(m+1)-m-i)	=  	p^{2h-1}  + m(p^{2h-1} -p^h) - j(p^h -p^{h-1})
			> 	p^{2h-1}  + m(p^{2h-1} -p^h) - (p^{h-1} - m)(p^h -p^{h-1})
			= m(p^{2h-1}-p^{h-1}) + p^{2h-2}
			\geq p^{2h-1} + p^{2h-2}-p^{h-1}
			>p^{2h-1}.
		\end{dmath*}
	This shows that the \(P_m(A,B)\) term vanishes modulo \((x^{p^{2h-1}+1})\) in the case \(m=j\), and therefore that
	\[P_m(A,B) \equiv 0\]
	for all \(m,j\).

	Next, we will analyze the \(C_{p^{h-1}}(A,B)\) term and the  \(C_{p^h}(A,B)\) term. The latter is divisible by the following power of \(x\)
		\[p^{h-1}(p^h-1) + p^h > p^{2h-1}.\]
		So it vanishes modulo \((x^{p^{2h-1}+1})\).
	
		We focus now on the \(C_{p^{h-1}}\) term. By Lemma \ref{Cpn},
		\begin{dmath*}
			C_{p^{h-1}}(A,B) 
			=\sum_{i=1}^{p-1}\frac{1}{p}{p\choose i} A^{p^{h-2}i}B^{p^{h-2}(p-i)}
			=\sum_{i=1}^{p-1}\frac{1}{p}{p\choose i} A^{p^{h-2}i}(\u^pB_0 + B_1)^{p^{h-2}(p-i)}.
		\end{dmath*}
		We note that
		 \[B_1^{p^{h-2}} =t_1^{p^{h-2}}x^{p^{2h-1}} \mod (x^{p^{2h-1}+1}). \] 
		 Since \(x\) divides \(A\), all terms involving \(B_1\) vanish modulo \((x^{p^{2h-1}+1})\).
		Whence \(\u^{p^{h-1}}\) divides \(C_{p^{h-1}}(A,B)\), and therefore \( \u C_{p^{h-1}}(A,B)\ =0 \).
		
		So far we have shown that, modulo \((x^{p^{2h-1}+1},\u^{p^{h-1}+1})\),
			\begin{dmath*}
			h_g([p]_{g_*F}(x)) = A + B 
		 =  t_0(g_*(\u)x^{p^{h-1}} \gF x^{p^h}) +[ t_1(g_*(\u)x^{p^{h-1}} + x^{p^h})^p + ... +  t_{h-1}(g_*(\u)x^{p^{h-1}} + x^{p^h})^{p^{h-1}}].
		 			\end{dmath*}
		The coefficient of \(x^{p^{2h-1}}\) here is
		\[c+ t_{h-1} \]
		where \(c\) is the coefficient of \(x^{p^{2h-1}}\) in 
		\[A=  t_0(g_*(\u)x^{p^{h-1}} \gF x^{p^h}).\]
		We'll show that \(c=0\).

		As before,  \[C_{p^h}(g_*(\u)x^{p^{h-1}},x^{p^h}) =0 \mod (x^{p^{2h-1}+1}).\]  
		
		We'll consider the \(P_m(g_*(\u)x^{p^{h-1}},x^{p^h})\) term now. Ignoring the integer coefficients, we  consider
		\begin{equation}\label{P}
			\sum_{m=1}^{p-1}\u^{m+1}  \sum_{j=0}^{m} (g_*(\u)x^{p^{h-1}} +x^{p^h})^{p^{h-1}(j+1)-m}(g_*(\u)^{p^{h-1}}x^{p^{2h-2}} + x^{p^{2h-1}})^{m-j}.
		\end{equation}
		
		As before, if \(m-j \neq 0\), all terms vanish modulo \((x^{p^{2h-1}+1},\u^{p^{h-1}+1})\). We set \(m=j\) in \eqref{P} to obtain
		\begin{dmath*}
			\sum_{m=1}^{p-1}\u^{m+1} (g_*(\u)x^{p^{h-1}} +x^{p^h})^{p^{h-1}(m+1)-m}
			= \sum_{m=1}^{p-1}\u^{m+1} \sum_{i=0}^{p^{h-1}(m+1)-m}{{p^{h-1}(m+1)-m} \choose {i} }   g_*(\u)^ix^{ip^{h-1}+p^h (p^{h-1}(m+1)-m-i)}.
		\end{dmath*}

		Since \(\u\) divides \(g_*(\u)\), the above sum will vanish modulo \((\u^{p^{h-1}+1})\) unless
		\[ i +m+1 < p^{h-1}+1.\]
		In this case, the exponent on \(x\) is
		\begin{dmath*}
			ip^{h-1} + p^h(p^{h-1}(m+1)-m-i) > 	ip^{h-1} + p^h(p^{h-1}(m+1)-p^{h-1}) = mp^{2h-1} + ip^{h-1} \geq p^{2h-1}.
		\end{dmath*}
	We have shown that
\[P_m(g_*(\u)x^{p^{h-1}},x^{p^h})\equiv 0.\]

		Using Lemma \ref{Cpn}, the \(C_{p^{h-1}}\) term is
		\begin{dmath*}
			C_{p^{h-1}}(g_*(\u)x^{p^{h-1}},x^{p^h}) = \sum_{i=1}^{p-1}\frac{1}{p} {p \choose i} g_*(\u)^{ip^{h-2}}x^{ip^{2h-3}+ p^{2h-2}(p-i)}.
		\end{dmath*}
		
		The exponent on \(x\) is
		\begin{dmath*} ip^{2h-3}+ p^{2h-2}(p-i) = i(p^{2h-3}-p^{2h-2}) + p^{2h-1} < p^{2h-1} . 
			\end{dmath*}
		
		So the \(C_{p^{h-1}}\) term does not contribute to the coefficient of \(x^{p^{2h-1}}\). Hence
		\[A=  t_0(g_*(\u)x^{p^{h-1}} + x^{p^h}),\]
		 \(c=0\) and \(A\) does not contribute to the coefficient of \(x^{p^{2h-1}}\). 
		
		We conclude that the coefficient on the left-hand side of \eqref{recursion} is
		\(t_{h-1}. \)
	\end{proof}

	\begin{restatable}{theorem}{th}\label{th} Assume \(h\geq 2\). Modulo \( (p,u_1,...,u_{h-2},\u^{p^{h-1}+1})\),
		\begin{dmath*}
			t_{h-1} 	=
			t_{h-1}^{p^{h}} + \u t_h^{p^{h-1}} 
			- \sum_{j=1}^{p-1}\frac{1}{p}{p \choose j}\u^{jp^{h-2}+1}t_1^{jp^{2h-3}}t_0^{p^{2h-2}(p-j)}.
		\end{dmath*}
	\end{restatable}	
	
	\begin{proof} The case \(h=2\) is treated by Lader in \cite{lader2013resolution}, so we assume \(h>2\).

			We equate the coefficients of \(x^{p^{2h-1}}\) calculated in Lemmas \ref{rhs} and \ref{lhs} to
	 complete the proof.
		\end{proof}

	\section{Formulas at height \(h=3\)}\label{h=3}
	
In this section we illustrate the utility of Theorems \ref{tk} and \ref{th} by computing an explicit approximation of \(t_0\) at height \(h=3\) for all primes \(p>2\).

We begin by analyzing our formulas for \(t_k\) and \(t_{h-1}\) to determine how accurately we can compute \(t_0\). Throughout this section we implicitly work modulo \(I_{h-1}\). 

Given an element \(g\in \SSS_h\) with expansion
\[g = 1 + g_1S + g_2S^2 + \cdots + g_{h-1}S^{h-1}+g_hS^h \mod (S^{h+1}),\]
we can express \(t_{h-1}\) in terms of \(g_1,g_{h-1}\), and \(g_h\) modulo \((\u^{p^{h-1}+1})\) by Theorem \ref{th}.
By Theorem \ref{tk},
\begin{dmath*}
	t_{h-2}= t_{h-2}^{p^h}+\u t_{h-1}^{p^{h-1}}-\u^{p^{h-1}}t_{h-1}t_0^{p^{h-1}(p^{h-1}-1)}.
\end{dmath*}
In the above expression,  the accuracy of \(t_{h-1}\) in the third term is the limiting factor. Therefore \(t_{h-2}\) can be expressed modulo \((\u^{2p^{h-1}+1})\). 

We see inductively 
that \(t_k\) can be computed modulo \((u^{2p^{h-1}+p^{h-2}+...+p^{k+1}+1})\). Hence the maximum accuracy possible to obtain with these formulas is \(u^{\Phi(h) + p^{h-1}}\) where
\[ \Phi(h) = p^{h-1} + p^{h-2} + ... + p +1.\]

This agrees with the accuracy of Corollary 3.4 of \cite{lader2013resolution}, and Corollary 4.5 of \cite{henn2013homotopy}.
		
With this in mind, we specialize to the case \(h=3\) and \(p>2\). We first untangle the recursion of Theorems \ref{tk} and \ref{th} to express \(t_0\) in terms of the coefficients of \(g\).

	\begin{lemma}\label{t0h=3lem} Assume \(h=3\) and \(p>2\). Let \(g = 1+g_1S+g_2S^2+g_3S^3 \mod (S^4)\). Then modulo \((p,u_1,\uu^{2p^2+p+1})\),
	\begin{align*}
		t_0 =& 1+ \uu(g_1^{p^2} +\uu^{p^2}g_2^p) \\
		&- \uu^p\left(g_1 + \uu (g_2^{p^2} + \uu^{p^2}g_3^p) - \uu^{p^2} \left( g_2 + \uu g_3^{p^2} - \sum_{j=1}^{p-1}\frac{1}{p}{{p} \choose {j}} \uu^{jp +1}g_1^{j}   \right) (1 + \uu^{p^2}g_1^p)^{p^2-1}\right)\\
		&\hspace{4cm} \left(1 + \uu^p g_1 - \uu^{p^2} (g_1^p + \uu^p g_2) (1+\uu^{p^2}g_1^p )^{p^2-1}  \right)^{p^2-1}.
	\end{align*}
\end{lemma}	
\begin{proof}
	Throughout the proof, we implicitly work modulo \((p,u_1)\), and all equalities are modulo \((\uu^{2p^2+p+1})\) unless otherwise stated.
	We remark that since \(p>2\),
	\[p^3 \geq 3p^2 > 2p^2 +p+1.\]
	
	To begin, we set \(h=3\) in Theorems \ref{tk} and \ref{th} and use that \(t_i = g_i \mod (p,u_1,\uu)\) and  \(g_i^{p^3}-g_i=0\) to obtain
	\begin{equation}\label{t0}
		t_0 = 1 + \uu t_1^{p^2} - \uu^p t_1 t_0^{p(p^2-1)}
	\end{equation}
	\begin{equation}\label{t1}
		t_1 = g_1 + \uu t_2^{p^2} - \uu^{p^2} t_2 t_0^{p^2(p^2-1)}
	\end{equation}
	\begin{equation}\label{t2}
		t_2 = g_2 + \uu g_3^{p^2} - \sum_{j=1}^{p-1}\frac{1}{p}{{p} \choose {j}} \uu^{jp +1}g_1^{j}  \mod (\uu^{p^2+1}).
	\end{equation}
	
	We will substitute \eqref{t1} and \eqref{t2} into \eqref{t0} repeatedly until we arrive at an expression for \(t_0\) in terms of \(g_1,g_2,\) and \(g_3\).

	We first record that
	\begin{equation}\label{t1p}
		t_1^p = g_1^p + \uu^p g_2
	\end{equation}
	\begin{equation}\label{t0p2}
		t_0^{p^2} = 1 + \uu^{p^2}g_1^p. 
	\end{equation}
	We substitute into \eqref{t0} to obtain
	\begin{align*}
		t_0 &= 1+ \uu(g_1^{p^2} +\uu^{p^2}g_2^p) - \uu^p\left(g_1 + \uu t_2^{p^2} - \uu^{p^2} t_2 t_0^{p^2(p^2-1)}\right) t_0^{p(p^2-1)}\\
		&= A - \uu^p\left(g_1 + \uu t_2^{p^2} - \uu^{p^2} t_2 (1 + \uu^{p^2}g_1^p)^{p^2-1}\right) t_0^{p(p^2-1)}\\
	\end{align*}
	where	
	\[ A =  1+ \uu(g_1^{p^2} +\uu^{p^2}g_2^p).\]

	By \eqref{t1p} and \eqref{t0p2},			
	\begin{align*}
		t_0^p 
		&=  1 + \uu^p g_1 - \uu^{p^2} t_1^p t_0^{p^{2}(p^2-1)} \\
		&=  1 + \uu^p g_1 - \uu^{p^2} (g_1^p + \uu^p g_2) \left(1+\uu^{p^2}g_1^p \right)^{(p^2-1)}.
	\end{align*}
	
	We record that	
	\begin{equation}\label{t2p2}
		t_2^{p^2} = g_2^{p^2} + \uu^{p^2}g_3^p \mod (\uu^{2p^2+p+1}). \end{equation}	
	
	This yields
	\begin{align*}
		t_0 &= A - \uu^p\left(g_1 + \uu (g_2^{p^2} + \uu^{p^2}g_3^p) - \uu^{p^2} t_2 (1 + \uu^{p^2}g_1^p)^{p^2-1}\right)\\
		&\hspace{4cm} \left(1 + \uu^p g_1 - \uu^{p^2} (g_1^p + \uu^p g_2) \left(1+\uu^{p^2}g_1^p \right)^{p^2-1}  \right)^{p^2-1}.
	\end{align*}
	
	Next, we use the fact that
	\begin{dmath*}
		\uu^{p^2+p}t_2
		= \uu^{p^2+p} (g_2 + \uu g_3^{p^2} - \sum_{j=1}^{p-1}\frac{1}{p}{{p} \choose {j}} \uu^{jp +1}g_1^{j}  ) \mod (\uu^{2p^2+p+1})
	\end{dmath*}
	to arrive at 
	\begin{align*}
		t_0 =& 1+ \uu(g_1^{p^2} +\uu^{p^2}g_2^p) \\
		&- \uu^p\left(g_1 + \uu (g_2^{p^2} + \uu^{p^2}g_3^p) - \uu^{p^2} \left( g_2 + \uu g_3^{p^2} - \sum_{j=1}^{p-1}\frac{1}{p}{{p} \choose {j}} \uu^{jp +1}g_1^{j}   \right) (1 + \uu^{p^2}g_1^p)^{p^2-1}\right)\\
		&\hspace{4cm} \left(1 + \uu^p g_1 - \uu^{p^2} (g_1^p + \uu^p g_2) (1+\uu^{p^2}g_1^p )^{p^2-1}  \right)^{p^2-1}.
	\end{align*}
\end{proof}	

Now that we have expressed \(t_0\) in terms of the \(g_i\)s, all that's left is to expand and collect like terms. 
For this, we will need the following lemma on binomial coefficients.
\begin{lemma}
	Let \(p\) be an odd prime, and let \(0\leq i \leq p^2-1 \). Then
	\begin{dmath*}
		{{p^2-1}\choose {i}} 
		= (-1)^i \mod (p)
	\end{dmath*}
\end{lemma}
\begin{proof}
	We claim that \[{{p-1} \choose {k}} = (-1)^k \mod (p)\] for all \(0\leq k \leq p-1\). The case \(k=0\) is obvious, so assume \(1\leq k \leq p-1\). We have the equality
	\[{{p} \choose {k}} = {{p-1}\choose{k-1}} + {{p-1}\choose {k}}. \] Since the left-hand side is zero, an inductive argument proves the claim.
	
	We write \(i\) as \(i = pi_1 + i_0\) where \(0 \leq i_0,i_1\leq p-1\). Then 
	\begin{align*}
		{{p^2-1}\choose {i}} &\equiv {{p-1} \choose {i_1}}{{p-1}\choose{i_0}} \\
		&\equiv (-1)^{i_1+i_0}\\
		&= (-1)^i. \qedhere
	\end{align*}
\end{proof}

\begin{restatable}{theorem}{th=3}
	\label{th=3}
	Assume \(h=3\) and \(p>2\). Let \(g = 1+g_1S+g_2S^2+g_3S^3 \mod (S^4)\). Then modulo \((p,u_1,\uu^{2p^2+p+1})\),
	\begin{dmath*}
		t_0
		=1+ \uu g_1^{p^2}
		- \sum_{i=0}^{p-1}(-1)^i \uu^{(i+1)p}g_1^{i+1} 
		-\sum_{i=0}^{p-2} (-1)^i\uu^{(i+1)p +1}g_1^ig_2^{p^2}
		+\uu^{p^2+1}\left(g_2^p - g_1^{p-1}g_2^{p^2}\right)
		- \sum_{i=0}^{p}\uu^{p^2+(i+1)p}g_1^{i}\left( (-1)^{p+i}g_1^{p+1} +  {{p^2-2}\choose{i}}g_1^{p+1} + (-1)^{i+1}g_2   +{{p^2-2}\choose {i-1}}g_2   \right) 
		- \uu^{p^2+p+1}\left(  g_3^p - g_3^{p^2} 
		\right)	
		-\sum_{j=1}^{p-1}\uu^{p^2+(j+1)p+1}g_1^{j-1} \left[C_j +\sum_{i=1}^j \frac{1}{p}{{p}\choose{i}}(-1)^{j-i}g_1   \right],
	\end{dmath*}	
	where
	\begin{dmath*}
		C_j =  (-1)^{j+p}g_1^{p+1}g_2^{p^2} + g_1\left( (-1)^j(g_3^p-g_3^{p^2}) + {{p^2-2}\choose{j}}g_1^pg_2^{p^2}  \right) + {{p^2-2} \choose{j-1}}g_2^{p^2+1}.
	\end{dmath*}
\end{restatable}

\begin{proof} 
	As before, we implicitly work modulo \((p,u_1)\), and all equalities are modulo \((\uu^{2p^2+p+1})\).
	
	We reference Lemma \ref{t0h=3lem} and first compute
	\begin{dmath*}
		\uu^{p^2}\left(1+\uu^{p^2}g_1^p \right)^{p^2-1} = \sum_{i=0}^{p^2-1}(-1)^i \uu^{(i+1)p^2}g_1^{pi}
		=\uu^{p^2}(1-\uu^{p^2}g_1^p).
	\end{dmath*}
	This yields
	\begin{align*}
		t_0 &= A - \uu^p\left(g_1 + \uu (g_2^{p^2} + \uu^{p^2}g_3^p) - \uu^{p^2} \left( g_2 + \uu g_3^{p^2} - \sum_{j=1}^{p-1}\frac{1}{p}{{p} \choose {j}} \uu^{jp +1}g_1^{j}   \right) (1-\uu^{p^2}g_1^p)\right)\\
		&\hspace{4cm} \left(1 + \uu^p g_1 - \uu^{p^2} (g_1^p + \uu^p g_2) (1-\uu^{p^2}g_1^p) \right)^{p^2-1}
	\end{align*}	
	where	
	\[ A =  1+ \uu(g_1^{p^2} +\uu^{p^2}g_2^p).\]

	We set
	\begin{dmath*}B
		= \uu^p\left(g_1 + \uu (g_2^{p^2} + \uu^{p^2}g_3^p) - \uu^{p^2} \left( g_2 + \uu g_3^{p^2} - \sum_{j=1}^{p-1}\frac{1}{p}{{p} \choose {j}} \uu^{jp +1}g_1^{j}   \right) (1-\uu^{p^2}g_1^p)\right)
		= \uu^pg_1 + \uu^{p+1}g_2^{p^2} + \uu^{p^2+p+1}g_3^p - \uu^{p^2+p}g_2 - \uu^{p^2+p+1} g_3^{p^2} + \sum_{j=1}^{p-1}\frac{1}{p}{{p} \choose {j}} \uu^{p^2+(j+1)p +1}g_1^{j}    +\uu^{2p^2+p}g_1^p  g_2      
	\end{dmath*}
	and
	\[C =  \left(1 + \uu^p g_1 + \uu^{p^2} (g_1^p + \uu^p g_2) (\uu^{p^2}g_1^p -1) \right)^{p^2-1}\]
	so that
	\[t_0 = A-BC.\]
	
	Next, we expand \(C\) as 
	\begin{align*}
		C	=& \sum_{i=0}^{p^2-1}(-1)^i (1+\uu^pg_1)^{p^2-1-i}\uu^{ip^2}(g_1^p + \uu^p g_2)^i (\uu^{p^2}g_1^p-1)^i \\
		=& (1+\uu^pg_1)^{p^2-1}\\
		&- (1+\uu^pg_1)^{p^2-2}\uu^{p^2}(g_1^p + \uu^p g_2) (\uu^{p^2}g_1^p-1)\\
		&+  (1+\uu^pg_1)^{p^2-3}\uu^{2p^2}(g_1^p + \uu^p g_2)^2 (\uu^{p^2}g_1^p-1)^2.
	\end{align*}		
	
	This yields
	
	\begin{dmath*}
		BC 
		= B(1+\uu^pg_1)^{p^2-1} 
		- \uu^{p^2}\left(\uu^pg_1 + \uu^{p+1}g_2^{p^2} - \uu^{p^2+p}g_2 \right)(1+\uu^pg_1)^{p^2-2}(-g_1^p-\uu^p g_2 + \uu^{p^2}g_1^{2p})
		+\uu^{2p^2+p}g_1^{2p+1} 
		= B(1+\uu^pg_1)^{p^2-1} -D
		+\uu^{2p^2+p}g_1^{2p+1} 
	\end{dmath*}		
	where
	\begin{align*}
		D=& \uu^{p^2}\left(\uu^pg_1 + \uu^{p+1}g_2^{p^2} - \uu^{p^2+p}g_2 \right)(1+\uu^pg_1)^{p^2-2}(-g_1^p - \uu^p g_2 + \uu^{p^2}g_1^{2p})\\
		=& -\left(\uu^{p^2+p}g_1^{p+1} + \uu^{p^2+p+1}g_1^{p}g_2^{p^2} - \uu^{2p^2+p}g_1^pg_2\right)(1+\uu^pg_1)^{p^2-2} \\
		&-\left(\uu^{p^2+2p}g_1g_2 + \uu^{p^2+2p+1}g_2^{p^2+1}   \right)(1+\uu^pg_1)^{p^2-2} \\
		&+ \uu^{2p^2+p}g_1^{2p+1}(1+\uu^pg_1)^{p^2-2}.
	\end{align*}	
	We next expand
	\[
	(1+\uu^p g_1)^{p^2-2} = \sum_{i=0}^{p^2-2}{{p^2-2}\choose {i}} \uu^{ip}g_1^i  \]
	
	and substitute to obtain
	\begin{dmath*}
		D =  - \sum_{i=0}^{p}{{p^2-2}\choose {i}} \uu^{p^2+(i+1)p}g_1^{p+i+1}
		- \sum_{i=0}^{p-1}{{p^2-2}\choose {i}} \uu^{p^2+(i+1)p+1}g_1^{p+i}g_2^{p^2}
		- \sum_{i=0}^{p-1}{{p^2-2}\choose {i}} \uu^{p^2+ (i+2)p}g_1^{i+1}g_2
		- \sum_{i=0}^{p-2}{{p^2-2}\choose {i}} \uu^{p^2+ (i+2)p+1}g_1^ig_2^{p^2+1} \\
		+  \uu^{2p^2+p}(g_1^{p}g_2
		+ g_1^{2p+1}).
	\end{dmath*}

	We next focus on \(B(1+\uu^pg_1)^{p^2-1}\). We expand as before to obtain
	
	\begin{align*}
		B(1+\uu^pg_1)^{p^2-1} 
		=& \sum_{i=0}^{2p}(-1)^i \uu^{(i+1)p}g_1^{i+1}
		+\sum_{i=0}^{2p-1}(-1)^i \uu^{(i+1)p+1}g_1^{i}g_2^{p^2}\\
		&-\sum_{i=0}^{p}(-1)^i \uu^{p^2+(i+1)p}g_1^{i}g_2
		+\sum_{i=0}^{p-1}(-1)^i \uu^{p^2+(i+1)p+1}g_1^{i}(g_3^p-g_3^{p^2})\\
		&+	\sum_{i=0}^{p-1}\sum_{j=1}^{p-1}\frac{1}{p}{{p} \choose {j}}(-1)^i {u_2}^{p^2+(i+j+1)p +1}g_1^{i+j}\\
		&+ \uu^{2p^2+p}g_1^{p}g_2.
	\end{align*}
	We pause to simplify the double sum above as
	\begin{align*}
		\sum_{i=0}^{p-1}\sum_{j=1}^{p-1}\frac{1}{p}{{p} \choose {j}}(-1)^i {u_2}^{p^2+(i+j+1)p +1}g_1^{i+j}  &=  \sum_{j=1}^{p-1}\sum_{i=0}^{p-1}\frac{1}{p}{{p} \choose {j}}(-1)^i {u_2}^{p^2+(i+j+1)p +1}g_1^{i+j} \\
		&=  \sum_{j=1}^{p-1}\sum_{i=0}^{p-j-1}\frac{1}{p}{{p} \choose {j}}(-1)^i {u_2}^{p^2+(i+j+1)p +1}g_1^{i+j} \\
		&=  \sum_{j=1}^{p-1}\sum_{i=1}^{p-j}\frac{1}{p}{{p} \choose {j}}(-1)^{i-1} {u_2}^{p^2+(i+j)p +1}g_1^{i+j-1}.
	\end{align*}

	We put everything together to obtain
	\begin{dmath*}
		t_0 = A - B(1+\uu^pg_1)^{p^2-1} + D - \uu^{2p^2+p}g_1^{2p+1}
		=1+ \uu(g_1^{p^2} +\uu^{p^2}g_2^p)
		- \sum_{i=0}^{2p}(-1)^i \uu^{(i+1)p}g_1^{i+1}
		-\sum_{i=0}^{2p-1}(-1)^i \uu^{(i+1)p+1}g_1^{i}g_2^{p^2}
		- \sum_{i=0}^{p} \uu^{p^2+(i+1)p}g_1^i\left({{p^2-2}\choose {i}}g_1^{p+1} -(-1)^ig_2\right)
		-\sum_{i=0}^{p-1} \uu^{p^2+(i+1)p+1}g_1^{i}\left((-1)^i(g_3^p-g_3^{p^2}) + {{p^2-2}\choose {i}}g_1^pg_2^{p^2}\right)
		-\sum_{j=1}^{p-1}\sum_{i=1}^{p-j}\frac{1}{p}{{p}\choose {j}}(-1)^{i-1} \uu^{p^2+(i+j)p+1}g_1^{i+j-1}
		- \sum_{i=0}^{p-1}{{p^2-2}\choose {i}} \uu^{p^2+ (i+2)p}g_1^{i+1}g_2
		- \sum_{i=0}^{p-2}{{p^2-2}\choose {i}} \uu^{p^2+ (i+2)p+1}g_1^ig_2^{p^2+1} \\
		= 1+ \uu g_1^{p^2} + \uu^p R + \uu^{p+1}S 
	\end{dmath*}
	where
	\begin{align*}
		\uu^p R = &- \sum_{i=0}^{2p}(-1)^i \uu^{(i+1)p}g_1^{i+1} 
		- \sum_{i=0}^{p} \uu^{p^2+(i+1)p}g_1^i\left({{p^2-2}\choose {i}}g_1^{p+1} -(-1)^{i}g_2\right) \\
		&- \sum_{i=0}^{p-1}{{p^2-2}\choose {i}} \uu^{p^2+ (i+2)p}g_1^{i+1}g_2 
	\end{align*}
	and
	\begin{dmath*}
		\uu^{p+1}S =  \uu^{p^2+1}g_2^p -\sum_{i=0}^{2p-1}(-1)^i \uu^{(i+1)p+1}g_1^{i}g_2^{p^2}
		-\sum_{i=0}^{p-1} \uu^{p^2+(i+1)p+1}g_1^{i}\left((-1)^i(g_3^p-g_3^{p^2}) + {{p^2-2}\choose {i}}g_1^pg_2^{p^2}\right)
		-\sum_{j=1}^{p-1}\sum_{i=1}^{p-j}\frac{1}{p}{{p}\choose {j}}(-1)^{i-1} \uu^{p^2+(i+j)p+1}g_1^{i+j-1} 
		- \sum_{i=0}^{p-2}{{p^2-2}\choose {i}} \uu^{p^2+ (i+2)p+1}g_1^ig_2^{p^2+1}. 
	\end{dmath*}
	
	Now we simplify to obtain
	\begin{dmath*}
		\uu^pR
		= - \sum_{i=0}^{p-1}(-1)^i \uu^{(i+1)p}g_1^{i+1}
		- \sum_{i=0}^{p}(-1)^{p+i} \uu^{p^2+(i+1)p}g_1^{p+i+1}
		- \sum_{i=0}^{p} \uu^{p^2+(i+1)p}g_1^i\left({{p^2-2}\choose {i}}g_1^{p+1} + (-1)^{i+1}g_2\right)  
		- \sum_{i=0}^{p}{{p^2-2}\choose {i-1}} \uu^{p^2+ (i+1)p}g_1^{i}g_2.
	\end{dmath*}
	We follow the convention that \({{m}\choose{n}}=0\) if \(n<0\). Therefore
	\begin{dmath*}
		\uu^pR	= - \sum_{i=0}^{p-1}(-1)^i \uu^{(i+1)p}g_1^{i+1} 
		- \sum_{i=0}^{p}\uu^{p^2+(i+1)p}g_1^{i}\left( (-1)^{p+i}g_1^{p+1} +  {{p^2-2}\choose{i}}g_1^{p+1} + (-1)^{i+1}g_2   +{{p^2-2}\choose{i-1}}g_2   \right).
	\end{dmath*}

	Next, we simplify \(\uu^{p+1}S\). Ignoring the double sum, we have
	\begin{dmath*}
		\uu^{p^2+1}g_2^p	-\sum_{i=0}^{2p-1}(-1)^i \uu^{(i+1)p+1}g_1^{i}g_2^{p^2}
		-\sum_{i=0}^{p-1} \uu^{p^2+(i+1)p+1}g_1^{i}\left((-1)^i(g_3^p-g_3^{p^2}) + {{p^2-2}\choose {i}}g_1^pg_2^{p^2}\right)
		- \sum_{i=0}^{p-2}{{p^2-2}\choose {i}} \uu^{p^2+ (i+2)p+1}g_1^ig_2^{p^2+1}\\
		=\uu^{p^2+1}g_2^p -\sum_{i=0}^{p-1}(-1)^i \uu^{(i+1)p+1}g_1^{i}g_2^{p^2}  
		 + \uu^{p^2+p+1}g_1^{p}g_2^{p^2}
		-\sum_{i=1}^{p-1}(-1)^{i+p} \uu^{p^2+(i+1)p+1}g_1^{p+i}g_2^{p^2}\\
		-\uu^{p^2+p+1}\left((g_3^p-g_3^{p^2}) + g_1^pg_2^{p^2}\right)
		-\sum_{i=1}^{p-1} \uu^{p^2+(i+1)p+1}g_1^{i}\left((-1)^i(g_3^p-g_3^{p^2}) + {{p^2-2}\choose {i}}g_1^pg_2^{p^2}\right)
		- \sum_{i=1}^{p-1}{{p^2-2}\choose {i-1}} \uu^{p^2+ (i+1)p+1}g_1^{i-1}g_2^{p^2+1}\\
		=  \uu^{p^2+1}(g_2^p-g_1^{p-1}g_2^{p^2})	-\sum_{i=0}^{p-2} (-1)^i\uu^{(i+1)p +1}g_1^ig_2^{p^2}
		- \uu^{p^2+p+1}\left( g_3^p - g_3^{p^2} 
		\right)	
		-\sum_{i=1}^{p-1}\uu^{p^2+(i+1)p +1}g_1^{i-1}C_i
	\end{dmath*}
	where
	\begin{dmath*}
		C_i =  (-1)^{i+p}g_1^{p+1}g_2^{p^2} + g_1\left( (-1)^i(g_3^p-g_3^{p^2}) + {{p^2-2}\choose{i}}g_1^pg_2^{p^2}  \right) + {{p^2-2} \choose{i-1}}g_2^{p^2+1}.
	\end{dmath*}
	
	We claim that
	\begin{dmath*}
		-\sum_{i=1}^{p-1}\uu^{p^2+(i+1)p+1}g_1^{i-1}C_i - \sum_{j=1}^{p-1}\sum_{i=1}^{p-j}\frac{1}{p}{{p}\choose{j}} (-1)^{i-1} \uu^{p^2+(i+j)p+1}g_1^{i+j-1} 
		=-\sum_{j=1}^{p-1}\uu^{p^2+(j+1)p+1}g_1^{j-1} \left[C_j +\sum_{i=1}^j \frac{1}{p}{{p}\choose{i}}(-1)^{j-i}g_1   \right].
	\end{dmath*}

	To prove the claim, it suffices to show
	\begin{equation}\label{double sum} \sum_{j=1}^{p-1}\sum_{i=1}^{p-j}\frac{1}{p}{{p}\choose{j}} (-1)^{i-1} \uu^{p^2+(i+j)p+1}g_1^{i+j-1} =\sum_{k=1}^{p-1}\sum_{r=1}^k \frac{1}{p}{{p}\choose{r}}(-1)^{k-r}\uu^{p^2+(k+1)p+1}g_1^{k}.
	\end{equation}
	We set \[k = i+j-1, \ r=j\] in the left-hand side of \eqref{double sum}. We observe that
	\[1 \leq k \leq p-j+j-1 = p-1\]
	and that \(1 \leq j \leq i+j-1\) implies \[1 \leq r \leq k.\] We substitute, then switch back to \(i,j\) to prove the claim.
	
	Therefore 
	\begin{dmath*}
		\uu^{p+1}S = 
		-\sum_{i=0}^{p-2} (-1)^i\uu^{(i+1)p +1}g_1^ig_2^{p^2}
		+\uu^{p^2+1}\left(g_2^p - g_1^{p-1}g_2^{p^2}\right)
		- \uu^{p^2+p+1}\left(  g_3^p - g_3^{p^2}
		\right)	
		-\sum_{j=1}^{p-1}\uu^{p^2+(j+1)p+1}g_1^{j-1} \left[C_j +\sum_{i=1}^j \frac{1}{p}{{p}\choose{i}}(-1)^{j-i}g_1   \right].
	\end{dmath*}
	
	We conclude that
	\begin{dmath*}
		t_0 = 1+ \uu g_1^{p^2} + \uu^p R + \uu^{p+1}S 
		=1+ \uu g_1^{p^2}
		- \sum_{i=0}^{p-1}(-1)^i \uu^{(i+1)p}g_1^{i+1} 
		- \sum_{i=0}^{p}\uu^{p^2+(i+1)p}g_1^{i}\left( (-1)^{p+i}g_1^{p+1} +  {{p^2-2}\choose{i}}g_1^{p+1} + (-1)^{i+1}g_2   +{{p^2-2}\choose {i-1}}g_2   \right) \\
		-\sum_{i=0}^{p-2} (-1)^i\uu^{(i+1)p +1}g_1^ig_2^{p^2}
		+\uu^{p^2+1}\left(g_2^p - g_1^{p-1}g_2^{p^2}\right)
		- \uu^{p^2+p+1}\left(  g_3^p - g_3^{p^2} 
		\right)	
		-\sum_{j=1}^{p-1}\uu^{p^2+(j+1)p+1}g_1^{j-1} \left[C_j +\sum_{i=1}^j \frac{1}{p}{{p}\choose{i}}(-1)^{j-i}g_1   \right].
	\end{dmath*}
We rearrange to complete the proof.		
\end{proof}

	\bibliographystyle{plain}

	\bibliography{references}
	
\end{document}